\documentclass[11pt]{article}
\usepackage{amsmath}
\usepackage{amsthm}
\usepackage{amsfonts}
\usepackage{float}
\usepackage{diagbox}
\usepackage[dvips,final]{graphicx}
\usepackage{dcolumn}
\usepackage{comment} 
\usepackage{xspace}
\usepackage{url}
\usepackage{upgreek}
\usepackage{multicol}
\usepackage{pict2e,multirow}
\usepackage{verbatim}
\usepackage{wrapfig}
\usepackage{url}
\usepackage{tikz}
\usepackage[colorlinks=true,citecolor=blue,urlcolor=blue]{hyperref}
\usepackage{subfigure}
\newtheorem{thm}{Theorem}

\newtheorem{cor}{Corollary}
\newtheorem{rmk}{Remark}

\newcommand{\be}{\begin{equation}}
\newcommand{\ee}{\end{equation}}

\newcommand{\Dx}{\Delta x}

\newcommand{\Dt}{\Delta t}

\newcommand{\dx}{\Delta x}
\newcommand{\dt}{\Delta t}
\newcommand{\aij}{\alpha_{i,j}}
\newcommand{\bij}{\beta_{i,j}}

\newcommand{\m}[1]{\mathbf{#1}}
\newcommand{\mA}{\m{A}}

\newcommand{\mS}{\m{S}}

\newcommand{\mR}{\m{R}}
\newcommand{\mP}{\m{P}}

\renewcommand{\v}[1]{\boldsymbol{#1}}

\newcommand{\vb}{\v{b}}
\newcommand{\vc}{\v{c}}
\newcommand{\ve}{\v{e}}
\newcommand{\vu}{\v{u}}

\newcommand{\va}{\v{a}}

\newcommand{\sspcoef}{\mathcal{C}}

\newcommand{\ceff}{\sspcoef_{\textup{eff}}}
\newcommand{\DtFE}{\Dt_{\textup{FE}}}
\newcommand{\tDtFE}{\tilde{\Dt}_{\textup{FE}}}

\renewcommand{\v}[1]{\mathbf{#1}}

\title{Strong Stability Preserving Integrating Factor Runge--Kutta Methods}

\author{%
Leah Isherwood\thanks{Mathematics Department, University of Massachusetts Dartmouth, 285 Old Westport Road,
North Dartmouth MA 02747.} \and
Zachary J. Grant\footnotemark[1] \and 
Sigal Gottlieb\footnotemark[1]
}

\begin{document}
\maketitle


\bibliographystyle{siam}

 {\small
{\bf Abstract.} Strong stability preserving (SSP)  Runge--Kutta methods are often desired when evolving in time  
problems that have two components  that have very different time scales. 
Where the SSP property is needed, it has been shown that  implicit and implicit-explicit methods 
have very restrictive time-steps and are therefore not efficient. For this reason, SSP
integrating factor methods may offer an attractive alternative to traditional time-stepping methods
for problems with  a linear component that is stiff and a  nonlinear component that is not.
However, the strong stability properties of integrating factor Runge--Kutta methods have not been established.
In this work we show that it is possible to define explicit integrating factor Runge--Kutta methods that 
preserve the desired strong stability properties satisfied by each of the two components when coupled with 
forward Euler time-stepping, or even given weaker conditions. 
We define sufficient conditions for an explicit integrating factor Runge--Kutta method
to be SSP, namely that they are based on explicit SSP Runge--Kutta methods with non-decreasing abscissas.
We find such methods of up to fourth order and up to ten stages,  analyze their SSP coefficients, and 
 prove their optimality  in a few cases. We test these methods to demonstrate their convergence and to show that
 the  SSP time-step predicted by the theory is generally sharp, and that the non-decreasing abscissa condition is needed
 in our test cases. Finally, we show that on  typical total variation diminishing linear and nonlinear test-cases our new
 explicit SSP integrating factor Runge--Kutta methods out-perform the corresponding explicit SSP Runge--Kutta methods,
 implicit-explicit SSP Runge--Kutta methods, and some well-known exponential time differencing methods.
 }

\section{Introduction\label{sec:intro}}

When numerically solving a hyperbolic partial differential  equation (PDE) of the form
\begin{eqnarray}
\label{PDE}
U_t +f(U)_x = 0,
\end{eqnarray}
the behavior of the numerical solution depends on properties of the spatial 
discretization combined with the time discretization.  
For smooth solutions, stability can be determined 
by analyzing the $L_2$ stability properties of the discretization applied to the linear problem.
However, when dealing with a non-smooth solution, stability in the $L_2$ norm is not sufficient to ensure
that the numerical solution will converge \cite{LeVequeBook}.  This is due to the presence of oscillations that prevents 
the approximation from converging uniformly.  To ensure that the numerical method
does not allow stability-destroying oscillations to form, we require that it satisfy stability properties
in, e.g,  the maximum norm or in the TV semi-norm.

Thus, to prove stability of numerical methods for nonlinear hyperbolic problems
with discontinuous solutions, we  need to analyze the nonlinear, non-inner-product stability properties
of a highly nonlinear,  complex spatial discretization combined with a high order time discretization. 
This is a difficult, sometimes untenable task. Instead, a method-of-lines formulation is generally followed, 
and a spatial discretization is developed that satisfies nonlinear, non-inner-product stability properties
when coupled with the forward Euler time stepping method.  In practice, higher order time discretizations
are needed. Strong stability preserving (SSP) time-discretizations were created \cite{shu1988b, shu1988} 
to allow the nonlinear non-inner product
stability properties of the spatial discretizations coupled with forward Euler to be 
immediately extended to all SSP higher order time-discretizations.

\subsection{Background\label{background}}
Linear stability theory is an indispensable tool used to establish the convergence of a numerical
method when numerically solving PDEs.  
Linear stability is necessary and sufficient 
for convergence  of a consistent linear 
numerical method when the PDE is  linear \cite{strikwerda1989}.  
When the PDE is nonlinear, if a numerical method is consistent and its {\em linearization} 
is $L_2$ stable and adequately dissipative, then convergence can be proved 
 for sufficiently smooth  problems  \cite{strang1964}.
However, discontinuous solutions often arise in the solution of hyperbolic conservation laws of the form
\eqref{PDE}, and when the solution has   discontinuities, 
linear stability theory is no longer sufficient for convergence.  

A famous example of this is that when the linearly $L_2$ stable second order Lax-Wendroff scheme 
is applied to Burgers'  equation, the method is nonlinearly unstable near stagnation points and does not converge \cite{majda1978}.  
This example demonstrates that to obtain convergence for  a nonlinear PDE with a discontinuous solution,
 some kind of nonlinear stability is necessary in order to guarantee convergence.  

Even in the linear case, $L_2$ stability is not enough for uniform convergence when dealing with solutions with discontinuities. 
For example, although a second order Lax-Wendroff scheme is strongly stable in the $L_2$ norm, when it is applied 
to a linear advection equation ($f(U)= a U$ in \eqref{PDE} above) with a step-function initial 
condition, the numerical solution will always have an overshoot or undershoot near the discontinuity  \cite{Lax2006}. 
Furthermore,  this is true not only for the second order Lax-Wendroff but indeed  
{\em any} linear, consistent, finite difference scheme  of at least second  order accuracy will develop a  
an overshoot or undershoot that prevents uniform convergence \cite{Lax2006}.

These two examples show that $L_2$ linear stability is not the relevant property when we desire well-behaved numerical 
solutions of hyperbolic PDEs with discontinuous solutions.
However, if we can prevent oscillations from forming by requiring stability in the maximum 
norm or the TV semi-norm, we can obtain uniform convergence \cite{LeVequeBook}. 
Consequently, a tremendous amount of effort
has been placed on the development of high order spatial discretizations
which, when coupled with the forward Euler time stepping method, have
the desired nonlinear stability properties for approximating discontinuous 
solutions of hyperbolic PDEs (see, e.g. \cite{harten1983, 
osher1984, sweby1984, cockburn1989, kurganov2000, tadmor1998, liu1994}).

However, for actual computation, higher order time discretizations 
are usually needed. There is no guarantee that a spatial discretization that 
is strongly stable in some desired norm or semi-norm 
(e.g., $L_{\infty}$, or $TV$) for a nonlinear problem under forward Euler 
integration will possess the same nonlinear stability property 
when coupled with a linearly stable higher order time discretization. 
Strong stability preserving methods were created to address this need.

\subsection{SSP methods\label{SSPintro}}
Explicit strong stability preserving (SSP) Runge--Kutta methods were  first developed in \cite{shu1988b,shu1988}
for use in conjunction with total variation diminishing (TVD) spatial discretizations  for 
hyperbolic conservation laws \eqref{PDE} with discontinuous solutions.
 These spatial discretizations of $f(U)_x$ ensure that when the resulting semi-discretized system of 
 ordinary differential equations (ODEs)
\begin{eqnarray}
\label{ode}
u_t = F(u),
\end{eqnarray}
is  evolved in time using the 
{\em forward Euler  method}, a strong stability property 
\begin{eqnarray} \label{FEstrongstability}
\|u^{n+1} \| = \| u^n + \dt F(u^{n}) \| \leq \| u^n \| 
\end{eqnarray}
is satisfied,  under some  step size restriction
\begin{eqnarray} \label{FEcond}
0 \leq \dt \leq \DtFE.
\end{eqnarray}
%
These TVD spatial discretizations are designed to satisfy the strong stability property 
\begin{eqnarray} \label{monotonicity}
\|u^{n+1} \| \leq \|u^{n} \| 
\end{eqnarray}
{\em when coupled with the forward Euler time discretization}. However, in actuality a 
higher order time integrator is desired, for both accuracy and linear stability reasons.
If we can  re-write a higher order time discretization as a convex combination of forward Euler steps,
we can ensure that any convex functional property \eqref{monotonicity} that is satisfied 
by the forward Euler method will still be satisfied by the  higher order time discretization, perhaps under a different time-step.

For example, we can write an  $s$-stage explicit Runge--Kutta method in Shu-Osher form \cite{SSPbook2011}:
\begin{eqnarray}
\label{rkSO}
u^{(0)} & =  & u^n, \nonumber \\
u^{(i)} & = & \sum_{j=0}^{i-1} \left( \aij u^{(j)} +
\dt \bij F(u^{(j)}) \right), \; \; \; \; i=1, . . ., s\\
 u^{n+1} & = & u^{(s)} . \nonumber
\end{eqnarray}
Note that for consistency, we must have $\sum_{j=0}^{i-1} \aij =1$.
If all the coefficients $\aij$ and $\bij$ are non-negative, and a given $\aij$ is zero only if its corresponding $\bij$ is zero,
then each stage can be rearranged into a convex combination of forward Euler steps
\[
\| u^{(i)}\|  =  
\left\| \sum_{j=0}^{i-1} \left( \aij u^{(j)} + \dt \bij F(u^{(j)}) \right) \right\|   \nonumber 
 \leq   \sum_{j=0}^{i-1} \aij  \, \left\| u^{(j)} + \dt \frac{\bij}{\aij} F(u^{(j}) \right\|  
 \leq  \|u^n\| , \]
where the final inequality follows from  \eqref{FEstrongstability} and  \eqref{FEcond} ,
provided that  the time-step satisfies 
\begin{eqnarray}
\dt \leq \min_{i,j} \frac{\aij}{\bij} \DtFE.
\end{eqnarray}
(Note that  if any of the $\beta$'s are equal to zero, the corresponding ratio is considered infinite).

It is clear from this example that  whenever we can re-write an explicit Runge--Kutta method as a  convex combination
of forward Euler steps, the forward Euler condition  \eqref{FEstrongstability}  will be {\em preserved} by the higher-order time 
discretizations, under the  time-step restriction 
$\Dt \le \sspcoef \DtFE$ 
where  $\sspcoef = \min_{i,j} \frac{\aij}{\bij}$.   As long as $ \sspcoef  >0$, the method is 
called {\em strong stability preserving} (SSP) with {\em SSP coefficient}
 $\sspcoef$ \cite{shu1988b}. 

This form also  ensures   internal stage strong stability,  i.e.,
$\| u^{(i+1)} \| \leq \| u^{(i)} \| $ at each stage $i$ of the time-stepping, under the same time-step restriction.
The internal stage monotonicity property is important in simulations that involve pressure, density, or water height,
in which a negative value even at the intermediate stage is not acceptable as it may not allow the simulation to proceed \cite{HesthavenCLbook}.
Positivity preserving limiters that prevent this from occurring are typically designed and proved for use with a forward Euler time-stepping
and thus naturally extend to SSP time stepping methods
with internal stage monotonicity
 \cite{HuAdamsShu2012,ZhangShu2010positivity,ZhangShu2011positivity, ZhangShu2012}.

We observe that in the original papers  \cite{shu1988b, shu1988}, 
the  term $\| \cdot \|$ in Equation  \eqref{FEstrongstability} above  represented
the total variation semi-norm. In general, though, the strong stability  preservation property holds
for any semi-norm, norm, or convex functional, as determined by the design of the spatial discretization.
The only requirements are that the forward Euler condition   \eqref{FEstrongstability}  holds, and that the time-discretization 
can be decomposed into a convex combination of forward Euler steps  with  $ \sspcoef  >0$, as above.

 Clearly, this convex combination condition is a  sufficient condition for strong stability preservation.
In fact, it has been shown that it is also necessary for  strong stability preservation
 \cite{SSPbook2011, kraaijevanger1991,spijker2007}. This means that 
if a method cannot be decomposed into a convex combination of forward Euler steps, then  we can always find 
some ODE with some initial condition such that the forward Euler condition is satisfied but the method
does not satisfy the strong stability condition for any positive time-step \cite{SSPbook2011}.
The observation that there are significant  connections between SSP theory and 
contractivity theory \cite{ferracina2004, ferracina2005,higueras2004a, higueras2005a} has led to 
many results on SSP methods as well as development  of 
new optimal and efficient SSP methods \cite{ferracina2008,
ketcheson2009, ketcheson2008}.  

It is not always possible to decompose a method into convex combinations of forward Euler steps where
$\sspcoef >0$. In fact in   \cite{kraaijevanger1991,ruuth2001} it was shown that explicit SSP Runge--Kutta methods  
cannot exist for order $p>4$. Furthermore, the value of $\sspcoef$ determines in large part what
the size of an allowable time-step will be, and so we seek methods that  have the largest possible
SSP coefficient. Of course, a more important quantity is the total cost of the time evolution,
 which is related to the allowable time step  relative to the number of function evaluations
at each time-step.  To allow us to compare the efficiency of explicit methods of a given order,
we  define the {\em effective SSP coefficient} $\ceff = \frac{\sspcoef}{s}$
where $s$ is the number of stages (typically the number of function evaluations). 
Unfortunately, all explicit $s$-stage Runge--Kutta 
methods have an SSP bound $\sspcoef \leq s$, and therefore $\ceff \leq 1 $ \cite{SSPbook2011}. Even worse,
this upper bound is not usually attained. However, many efficient explicit SSP Runge--Kutta methods have been found, as
 we discuss in Section \ref{sec:background}.

For smooth problems it is usually the case that implicit methods  (or implicit treatment of stiff terms)
can remove the time-step restriction needed for stability.
In such cases, where the timestep is limited by a linear stability requirement or by an 
inner-product norm nonlinear stability there are  well-known classes of implicit methods (e.g. A-stable, L-stable, B-stable) that
allow the use of arbitrarily large timesteps.  This is not the case for time discretizations of order $p>1$ where the time-step is limited by 
SSP considerations. For first order ($p=1$) it can be shown that if the  spatial discretization is strongly stable in some norm under
forward Euler time integration, then the fully discrete solution will also be strongly stable, in the same norm,
for the implicit Euler method, without any timestep restriction 
\cite{hundsdorfer2003, higueras2004a}. However,  any 
general linear method of order $p>1$ can only be SSP under some finite timestep \cite{spijker1983}.  
Worse yet, the timestep restrictions
for implicit and implicit-explicit (IMEX) SSP methods are not dramatically larger than those for explicit
methods, with the observed  bound of $\sspcoef \leq 2s$, and therefore $\ceff \leq 2 $  
\cite{lenferink1991,ferracina2008,ketcheson2009,SSPIMEX}.

SSP methods are widely used in the solution of hyperbolic PDEs,
with discontinuous solutions, where  linear $L_2$ stability is not enough to ensure convergence.
They have been paired with many spatial discretizations that were specially designed for hyperbolic PDEs.
Some of these spatial discretization approaches  involve numerical methods that directly incorporate 
a non-oscillatory approach, such as ENO \cite{caiden2001,delzanna2002,baiotti2005} and WENO  
\cite{bassano2003,carrillo2003,tanguay2003,feng2004,labrunie2004,balbas2005,zhang2006,pantano2007}
methods,  in a finite difference or finite element setting.
These methods are typically designed and their properties proved with a forward Euler time-stepping, and rely on 
the SSP time-discretization mechanism to extend to higher order in time.
Other approaches include limiters applied to finite difference or discontinuous Galerkin methods  \cite{cockburn2004}.
These limiters may enforce a total variation diminishing (TVD) \cite{sweby1984} or total variation 
bounded (TVB) \cite{shu1987, cockburn1989} solution, or be used to ensure the numerical solution is maximum principle preserving 
\cite{ZhangShu2010maximum, ZhangShu2011maximum}, or positivity preserving
 \cite{ZhangShu2010positivity,ZhangShu2011positivity,ZhangShu2012}.
In these cases, SSP methods have proven particularly popular because for each new limiter proposed 
proofs of the desired  property are generated  for  the method coupled with the first order
forward Euler time-discretization.  Once again, the SSP time-discretization mechanism is needed to extend 
the results on these limiters to higher order in time.

For spectral and pseudospectral methods, Reddy and Trefethen \cite{ReddyTrefethen1990} explored the fact that eigenvalue stability is insufficient
to ensure stable simulations, and discussed the difficulties in a fully-discrete stability analysis.
Gottlieb and Tadmor \cite{GottliebTadmor1991} proved stability of spectral approximations for the forward Euler method;
This is the approach commonly used in spectral methods simulations \cite{HGGbook}. 
The work of Levy and Tadmor \cite{LevyTadmor1998} shows the challenges of going from analysis of semi-discrete stability to
fully discrete stability, as they analyzed the strong stability of Runge--Kutta schemes (including the first order forward Euler method)
for linear problems. 
Gottlieb, Shu, and Tadmor later showed that the stability of the forward Euler method for any linear coercive approximations
 \cite{LevyTadmor1998}
could be easily extended using SSP analysis to a much larger class of Runge-Kutta methods. Furthermore, the SSP approach   
guarantees stability for much larger CFL numbers than the stability analysis  in  \cite{LevyTadmor1998}.
When using spectral methods on nonlinear problems, regularization using filtering \cite{HesthavenGottlieb2001,HGGbook} 
or specially designed viscosity \cite{MadayTadmor1989,Tadmor1990, Ma1989,HGGbook} is needed for the 
simulation to be stable. The stabilization properties of these techniques are proven, 
usually, first on the semi-discrete form, and only then
for the fully discrete form,  in conjunction with a forward Euler step;  the
SSP time-stepping mechanism allows it to be preserved for higher order methods \cite{HGGbook, HesthavenCLbook}.

In addition, SSP time-stepping methods  have been paired 
with other spatial discretizations, including level set methods
\cite{peng1999,caiden2001,enright2002,cheng2003,cockburn2005,jin2005}, 
spectral finite volume methods \cite{sun2006,cheruvu2007}, and
spectral difference methods \cite{wang2005,wang2007a}.  
Examples of application areas where SSP time-stepping methods have been used include:
compressible flow \cite{wang2005}, incompressible flow \cite{patel2005},
viscous flow \cite{sun2006},
two-phase flow \cite{caiden2001,bassano2003}, relativistic flow 
\cite{delzanna2002,baiotti2005,zhang2006}, 
cosmological hydrodynamics \cite{feng2004},
magnetohydrodynamics \cite{balbas2005}, radiation hydrodynamics
\cite{mignone2005}, two-species plasma flow \cite{labrunie2004},
atmospheric transport \cite{cheruvu2007}, large-eddy simulation 
\cite{pantano2007}, Maxwell's equations \cite{cockburn2004}, 
semiconductor devices \cite{carrillo2003}, lithotripsy \cite{tanguay2003},
geometrical optics \cite{cockburn2005}, and
Schrodinger equations \cite{cheng2003,jin2005}.

 \subsection{SSP integrating factor Runge--Kutta methods}
 
 In this work, we are interested in a semi-discretized problem of the form
 \[ u_t = Lu + N(u) \]
 where each component satisfies a forward Euler condition 
\[ 
 \| u^n + \dt L u^n \| \leq \|  u^n\|   \; \; \;   \mbox{for} \; \; \; \dt \leq \tilde{\DtFE}  \]
and 
\[  \| u^n + \dt N(u^n) \| \leq |\ u^n\| \; \; \;   \mbox{for}  \; \; \; \dt \leq {\DtFE} , \]
 (where $\| \cdot \|$ is some convex functional needed for non-linear, non-inner-product stability).
 In the cases of interest, $\tilde{\DtFE} <<  {\DtFE} $, so that 
 $L$ is a linear operator that significantly restricts the allowable time-step.
 As mentioned above, when the SSP properties are a concern, using an implicit-explicit (IMEX) scheme does not
 significantly alleviate the allowable time-step \cite{SSPIMEX}.
This motivates our investigation of integrating factor methods, where the 
 linear component $L u$ is handled exactly, and then 
 the time-step restriction is, at worst, $ \dt \leq {\DtFE}$ coming from the  the nonlinear component $N(u)$. 
 In this work, we discuss the conditions under which 
this process  guarantees that the strong stability property \eqref{monotonicity} is preserved. 
In particular, we show that if we  step the transformed problem forward using 
an SSP Runge--Kutta method where
the abscissas (i.e. the time-levels approximated by each stage) are non-decreasing, 
we obtain a method that preserves the desired  strong stability property.
It is important to note that the efficient use of the proposed methods will depend
heavily on the cost of computation of the matrix exponential. This area of research,
although outside the scope of this work, will be crucial for the practical implementation of
the SSP integrating factor Runge--Kutta methods described in this paper.

The paper is structured as follows: In Section \ref{sec:background} we review 
some known optimal and optimized explicit SSP Runge--Kutta methods of orders $p \leq 4$ and give
the SSP coefficients and effective SSP coefficients of the optimal methods in this class. 
In Section \ref{sec:SSPIF} we describe explicit  integrating factor (also known as Lawson type) Runge--Kutta (IFRK) methods
\cite{lawson1967}.
We prove that when IFRK methods  are based on 
explicit SSP Runge--Kutta methods with  {\em non-decreasing} abscissas, they preserve the strong stability property. 
We also show an example that demonstrates that when using an 
IFRK method based on 
an explicit SSP Runge--Kutta method that has decreasing abscissas, 
the SSP property is violated. 
In Section \ref{sec:optimization} we formulate the optimization problem that will enable us to find
optimized explicit SSP Runge--Kutta methods that have non-decreasing abscissas, and in 
Section  \ref{sec:optimal} we present some optimized methods in this class and their SSP coefficients
and effective SSP coefficients. We also prove the optimality of one of the methods in this class.
In Section  \ref{sec:test} we present numerical examples that show how our explicit SSP
integrating factor Runge--Kutta (eSSPIFRK) methods perform on typical test cases compared to explicit, implicit-explicit (IMEX),
and exponential time differencing (ETD) methods. We also compare the linear stability properties of these methods.

In Section \ref{sec:conclusions} we conclude that  the newly developed SSP theory for integrating factor 
Runge--Kutta methods provides a provable bound on the allowable time-step   (which is often sharp in practice)
for preservation of the nonlinear non-inner-product stability properties. Moreover, the
newly developed methods demonstrate a  significantly larger allowable SSP step-size than standard methods 
including the exponential time-differencing methods of  Cox and Matthews \cite{CoxMatthews2002}, 
 SSP IMEX methods \cite{SSPIMEX}, standard explicit SSP Runge--Kutta methods \cite{SSPbook2011}, 
as well as the Runge--Kutta methods of Kinnmark and Grey \cite{KG1984}.

\bigskip

\noindent{\bf Note:} The following acronyms and notations are used in this work: 

\smallskip
\noindent
\begin{tabular}{|ll|} \hline
IFRK & integrating factor Runge--Kutta method.\\
SSP & strong stability preserving. \\
eSSPRK & explicit SSP Runge--Kutta method. \\
eSSPRK$^+$ & explicit SSP Runge--Kutta  method with non-decreasing abscissas. \\
eSSPKG & explicit SSP  Kinnmark and Gray method. \\
eSSPRK$^+$ & explicit SSP  Kinnmark and Gray method with  \\
& non-decreasing abscissas. \\
eSSPIFRK & explicit SSP  integrating factor Runge--Kutta method. \\
(s,p) & number of stages $s$ and order $p$. \\
IMEX & implicit-explicit additive method. \\
ETD & exponential time-differencing methods.\\
\hline
\end{tabular}

\section{A review of explicit SSP Runge--Kutta methods} \label{sec:background}
SSP Runge--Kutta methods guarantee the strong stability (in any norm, semi-norm, or convex functional)
of the numerical solution of any ODE  provided {\em only} that the forward Euler condition \eqref{FEstrongstability}  
is satisfied under a time step restriction  \eqref{FEcond}.
This  requirement  leads to severe restrictions on the allowable order of SSP methods,
and the allowable time step $\Dt \le \sspcoef \DtFE.$ 
These methods have been extensively studied, e.g., in 
\cite{ferracina2004, ferracina2005,ferracina2008,SSPbook2011,gottliebshu1998,
gottlieb2001, higueras2004a,higueras2005a, hundsdorfer2003, ketcheson2008, ketcheson2009,
ketcheson2009a,  ketcheson2011, KubatkoKetcheson, ruuth2001}.
In this section, we review some popular
and efficient explicit SSP Runge--Kutta methods, and present the SSP coefficients of optimized methods
of up to ten stages and fourth order. 


In the original papers on SSP time-stepping methods (there called TVD time-stepping)
\cite{shu1988b, shu1988},  the authors presented the first explicit SSP Runge--Kutta methods. 
These methods were  second and third order with SSP coefficient $\sspcoef=1$ ($\ceff=\frac{1}{2}$ and $\ceff=\frac{1}{3}$, respectively),
and were proven optimal \cite{gottliebshu1998}. We use the notation eSSPRK(s,p) to denote an explicit SSP Runge--Kutta method with
$s$ stages and of order $p$.

\noindent{\bf eSSPRK(2,2):}
\begin{eqnarray}
u^{(1)} & = &u^{n} + \Dt F(u^{n}) \nonumber \\
u^{n+1} & = &\frac{1}{2} u^{n} + \frac{1}{2} \left( u^{(1)} +  \Dt F(u^{(1)}) \right),
\end{eqnarray}
\noindent{\bf eSSPRK(3,3):}
\begin{eqnarray} \label{SSPRK33}
     u^{(1)} &= & u^n + \dt F(u^n) \nonumber \\
     u^{(2)} &= & \frac{3}{4} u^n + \frac{1}{4} \left(u^{(1)} + \dt F(u^{(1)})\right)  \nonumber \\
     u^{n+1} & = & \frac{1}{3} u^n + \frac{2}{3} \left(u^{(2)} +  \dt F(u^{(2)})\right).
\end{eqnarray}
Method \eqref{SSPRK33} has been extensively used and is known as the Shu-Osher method.
   \smallskip

No four stage fourth order explicit Runge--Kutta methods exist with a positive SSP coefficient
 \cite{gottliebshu1998,ruuth2001}. However, fourth order methods with more than four stages
   ($s>p$) do exist.

   \noindent{\bf eSSPRK(5,4):}
Found by Spiteri and Ruuth \cite{SpiteriRuuth2002} 
   \begin{eqnarray*}
u^{(1)} & = & u^n +  0.391752226571890 \dt F(u^n) \\ 
u^{(2)} & = &  0.444370493651235 u^n +  0.555629506348765 u^{(1)} 
+ 0.368410593050371 \dt F(u^{(1)}) \\ 
u^{(3)} & = &  0.620101851488403 u^n +  0.379898148511597 u^{(2)} 
 + 0.251891774271694  \dt F(u^{(2)}) \\ 
u^{(4)} & = &  0.178079954393132 u^n + 0.821920045606868 u^{(3)} 
+  0.544974750228521 \dt F(u^{(3)})\\ 
u^{n+1} & = &    0.517231671970585 u^{(2)} 
 +  0.096059710526147 u^{(3)} +  0.063692468666290 \dt F(u^{(3)}) \\ 
& & +  0.386708617503268 u^{(4)} +   0.226007483236906 \dt F(u^{(4)}) \, ,
\end{eqnarray*}
has $\sspcoef=1.508$ ($\ceff=0.302$).

   \noindent{\bf eSSPRK(10,4):} Found by
Ketcheson  \cite{ketcheson2008},  has a low-storage formulation:
\begin{eqnarray*}
u^{(1)} & = & u^n + \frac{1}{6} \dt F(u^n) \\ 
u^{(i+1)} & = & u^{(i)} + \frac{1}{6} \dt F(u^{(i)}) \; \; \;  i=1,2, 3\\ 
u^{(5)} & = & \frac{3}{5} u^n +   \frac{2}{5} \left(u^{(4)} +  \frac{1}{6} \dt F(u^{(4)})\right) \\ 
u^{(i+1)} & = & u^{(i)} + \frac{1}{6} \dt F(u^{(i)}) \; \; \; i=5,6,7,8 \\ 
u^{n+1} & = &   \frac{1}{25} u^{n} +  \frac{9}{25} \left(u^{(4)} + \frac{1}{6}  \dt F(u^{(4)}) \right)+  \frac{3}{5} \left(u^{(9)}
+ \frac{1}{6}  \dt F(u^{(9)})\right)  \, ,
\end{eqnarray*}
has  $\sspcoef=6$ ($\ceff=0.6$).

 \begin{table}
\centering
\setlength\tabcolsep{4pt}
\begin{minipage}{0.48\textwidth}
\centering
\begin{tabular}{|c|lll|} \hline
  \diagbox{s}{p}  &  2 & 3 & 4 \\
 \hline
    1  &  -           &  -           &   -\\   
    2  &  1.0000 &  -           &   -\\    
    3  &  2.0000 &  1.0000 &   -\\
    4  &  3.0000 &  2.0000 &   -\\
    5  &  4.0000 &  2.6506 &   1.5082\\
    6  &  5.0000 &  3.5184 &   2.2945\\
    7  &  6.0000 &  4.2879 &   3.3209\\
    8  &  7.0000 &  5.1071 &   4.1459\\
   9  &  8.0000 &  6.0000 &   4.9142\\
   10 &  9.0000 &  6.7853 &   6.0000\\
\hline
\end{tabular}
\caption{SSP coefficients of the optimized eSSPRK(s,p) methods \cite{SSPbook2011}.}
\label{tab:SSPcoef} 
\end{minipage}%
\hfill
\begin{minipage}{0.48\textwidth}
\centering
\begin{tabular}{|c|lll|} \hline
  \diagbox{s}{p}  &  2 & 3 & 4 \\
 \hline
    1  &   -           &   -           &   -\\   
    2  &   0.5000 &   -           &   -\\
    3  &   0.6667 &   0.3333 &   -\\
    4  &   0.7500 &   0.5000 &   -\\
    5  &   0.8000 &   0.5301 &   0.3016\\
    6  &   0.8333 &   0.5864 &   0.3824\\
    7  &   0.8571 &   0.6126 &   0.4744\\
    8  &   0.8750 &   0.6384 &   0.5182\\
    9  &   0.8889 &   0.6667 &   0.5460\\
   10 &   0.9000 &   0.6785 &   0.6000\\
\hline
\end{tabular}
\caption{Effective SSP coefficients of the optimized eSSPRK(s,p) methods \cite{SSPbook2011}.}
 \label{tab:effSSPcoef} 
\end{minipage}
\end{table}

In Table \ref{tab:SSPcoef} we present the SSP coefficients of optimized explicit SSP Runge--Kutta methods of up to $s=10$ stages and
order $p=4$, and in  Table \ref{tab:effSSPcoef}  the corresponding effective SSP coefficients.
Unfortunately, no methods of order $p \geq 5$ with positive SSP coefficients can exist \cite{kraaijevanger1991,ruuth2001}.

\section{Explicit SSP Runge--Kutta  schemes for use with  integrating factor methods} \label{sec:SSPIF} 
We consider a problem of the form
\begin{eqnarray}
u_t = Lu + N(u)  
\end{eqnarray}
with a linear constant coefficient component $L u $ and a nonlinear component $N(u)$.
The case we are interested in is when some strong stability condition is known for the forward Euler step of the 
nonlinear component $N(u)$
\begin{eqnarray}  \label{nonlinearFEcond}
\| u^n + \dt N(u^n) \| \leq \|u^n\| \qquad \mbox{for} \qquad \dt \leq \DtFE
\end{eqnarray}
while taking a forward Euler step using the linear component $Lu$ results in the  strong stability condition 
\begin{eqnarray} \label{linearFEcond}
\| u^n + \dt L u^n \| \leq \|u^n\| \qquad \mbox{for} \qquad \dt \leq \tDtFE
\end{eqnarray}
where $ \tDtFE <<  \DtFE$. 
In such cases, stepping forward using an explicit SSP Runge--Kutta method, or even an implicit or an  implicit-explicit (IMEX)
SSP Runge--Kutta method will result in severe constraints on the allowable time-step \cite{lenferink1991,ketcheson2009,ferracina2008,SSPbook2011,SSPIMEX}.

An alternative methodology  that may alleviate the restriction on the allowable time-step involves solving the linear part exactly
using an integrating factor approach 
\begin{eqnarray*}
e^{-L t } u_t - e^{-L t }  Lu = e^{-L t } N(u)  \\
\left( e^{-L t } u \right)_t = e^{-L t } N(u).
\end{eqnarray*}
A transformation of variables $w = e^{-L t } u $ gives the ODE system
\begin{eqnarray} \label{IF_ODE}
w_t  = e^{-L t } N(e^{L t } w)   = G(w),
\end{eqnarray}
which can then be evolved forward in time using, for example, an explicit Runge--Kutta method of the form \eqref{rkSO}.
For each stage $u^{(i)}$, which corresponds to the solution at time $t_i = t^n + c_i \dt$ (where each $c_i$ is the abscissa of the method
at the $i$th stage), 
the  corresponding integrating factor Runge--Kutta method becomes 
\[
e^{-L t_i} u^{(i)}   =   \sum_{j=0}^{i-1} \left( \aij e^{-L t_j}  u^{(j)} + \dt \bij e^{-L t_j}  N(u^{(j)}) \right),  \]
or
\begin{eqnarray*}
u^{(i)}   &=&    \sum_{j=0}^{i-1} \left( \aij e^{L (t_i-t_j) }  u^{(j)} + \dt \bij e^{L ( t_i-t_j)}  N(u^{(j)}) \right) \\
 &=&  \sum_{j=0}^{i-1} \left( \aij e^{L (c_i-c_j) \dt }  u^{(j)} + \dt \bij e^{L ( c_i-c_j)\dt}  N(u^{(j)}) \right) . 
 \end{eqnarray*}

\noindent In the following results we establish the SSP properties of this approach.

\begin{thm} \label{thm:exp_FEcond}
If  a  linear operator $L$  satisfies  \eqref{linearFEcond} for some value of $\tDtFE > 0 $,
then 
\begin{equation} \label{EXPcondition}
\| e^{\tau L} u^n \|  \leq \| u^n \|   \; \; \; \; \forall \; \tau \geq 0 . 
\end{equation}
\end{thm}

\begin{proof}
The Taylor series expansion of $e^z$ can be written as
\[e^z = \sum_{j=0}^\infty \gamma_j(r)  \left(1+ \frac{z}{r}  \right)^j  \; \; \; 
\mbox{where} \; \; \gamma_j = \frac{r^j}{j!}  e^{-r}  
\]
where the coefficients $\gamma_j$  are clearly nonnegative for all values of $r \geq 0$.
These coefficients  sum to one because  
\[  \sum_{j=0}^\infty \gamma_j =  \sum_{j=0}^\infty \frac{r^j}{j!}   e^{-r}   = 
e^{-r}    \sum_{j=0}^\infty \frac{r^j}{j!} =
e^{-r} e^r = 1. \]
Using this  we can show that $e^{\tau L}  u^n$ can be written as a convex combination of forward Euler steps with a modified time-step
$ \frac{\tau}{r} $,
so that
\begin{eqnarray*} \label{expFEcond}
\| e^{\tau L} u^n \| & = & \left\| \sum_{j=0}^\infty \gamma_j(r)  \left(1+ \frac{\tau}{r} L \right)^j u^n \right\| \\
& \leq & \sum_{j=0}^\infty \gamma_j(r)   \left\| \left(1+ \frac{\tau}{r} L \right)^j u^n \right\| \\
& \leq & \sum_{j=0}^\infty \gamma_j(r)   \left\|  u^n \right\| 
 \leq  \left\| u^n \right\| \; \; \mbox{for any $0 \leq  \tau  \leq r  \tDtFE$}.
\end{eqnarray*}
As this is true for any value of $r \geq 0$, we have
\[\| e^{\tau L} u^n \|  \leq \| u^n \|   \; \; \; \; \forall \tau \geq 0 . \]
Note that a negative value of $\tau$ is not allowed here. 
\end{proof}

\begin{rmk} The thm above deals with the case that 
\eqref{linearFEcond} is satisfied for some value of $\tDtFE > 0 $. However,
requiring $L$ to satisfy only condition  \eqref{EXPcondition}
 is sufficient for  the integrating factor 
Runge--Kutta method to be SSP. In the following results, therefore,
we only require the condition \eqref{EXPcondition}, which is 
a weaker condition than \eqref{linearFEcond}. 
\end{rmk}

\begin{cor}
Given a  linear operator $L$ that satisfies  \eqref{EXPcondition}
and a (possibly nonlinear) operator $N(u)$ that satisfies \eqref{nonlinearFEcond}  
for some value of ${\Delta t}_{FE} > 0 $, we have 
\begin{equation}
\| e^{\tau L} ( u^n  + \dt N (u^n )) \|  \leq \| u^n \|   \; \; \; \; \forall \dt \leq \DtFE,  \; \; \; \mbox{provided that} \; \; \tau \geq 0. 
\end{equation}
\end{cor}
\begin{proof}
Separate the term $e^{\tau L} ( u^n  + \dt N(u^n) )$ to two steps:
\begin{eqnarray*}
y^{(1)} & = & u^n  + \dt N( u^n) \\
y^{(2)} & = &  e^{\tau L}  y^{(1)}
\end{eqnarray*}
Clearly, from \eqref{EXPcondition} we have 
\[ \| y^{(2)} \| = \|  e^{\tau L}  y^{(1)} \|  \leq \|   y^{(1)} \| \]
for any $\tau \geq 0$. Now, from  \eqref{nonlinearFEcond}  we also have 
\[ \| y^{(1)} \|  = \| u^n  + \dt N(u^n) \| \leq \|u^n\| \; \; \; \forall \dt \leq \DtFE .\]
Putting these two together we obtain the desired result.
\end{proof}

The following thm describes the conditions under which an integrating factor Runge--Kutta method
is strong stability preserving.
\begin{thm} \label{thm:SSPIF}
Given a  linear operator $L$ that satisfies   \eqref{EXPcondition}
and a (possibly nonlinear) operator $N(u)$ that satisfies \eqref{nonlinearFEcond}  
for some value of ${\Delta t}_{FE} > 0 $, and a Runge--Kutta integrating factor method of the form
\begin{eqnarray} \label{rkIFSO}
u^{(0)} & =  & u^n, \nonumber \\
u^{(i)} & = & \sum_{j=0}^{i-1} e^{L (c_i-c_j) \dt }  \left( \aij u^{(j)} + \dt \bij N(u^{(j)}) \right), \; \; \; \; i=1, . . ., s\\
 u^{n+1} & = & u^{(s)}  \nonumber
\end{eqnarray}
where $0=c_1 \leq c_2 \leq  . . . \leq c_s$, then $u^{n+1}$ obtained from \eqref{rkIFSO} satisfies
\begin{equation}
\|u^{n+1}\| \leq \|u^n\| \; \; \; \forall \dt \leq \sspcoef \DtFE.
\end{equation}
\end{thm}
\begin{proof}
We observe that  for each stage of \eqref{rkIFSO}
\begin{eqnarray*}
\| u^{(i)} \|& = & \left\| \sum_{j=0}^{i-1} e^{L (c_i-c_j) \dt }  \left( \aij u^{(j)} + \dt \bij N(u^{(j)}) \right) \right\| \\
& \leq  & \sum_{j=0}^{i-1} \left\|  e^{L (c_i-c_j) \dt }  \left( \aij u^{(j)} + \dt \bij N(u^{(j)}) \right)\right\| \\
& \leq  &  \sum_{j=0}^{i-1} \aij \left\|  e^{L (c_i-c_j) \dt }  \left( u^{(j)} + \dt \frac{\bij}{\aij} N(u^{(j)}) \right) \right\| 
\end{eqnarray*}
where the last inequality follows from Corollary 1, as long as $c_i - c_j \geq 0$ and 
$\dt \frac{\bij}{\aij}  \leq \DtFE$. This establishes the result of the thm.
Furthermore, this proof ensures that these methods have internal stage strong stability as well,  i.e.
$\| u^{(i+1)} \| \leq \| u^{(i)} \| $ at each stage $i$ of the time-stepping, under the same time-step restriction.
\end{proof}

\begin{rmk}
It is possible to preserve the strong stability property even with decreasing abscissas,
provided that whenever the term $c_i-c_j$ is negative, the operator $L$ is replaced by an operator
$\tilde{L}$ that satisfies the condition \[\| e^{-\tau \tilde{L}} u^n \|  \leq \| u^n \|   \; \; \; \; \forall \; \tau \geq 0 . \]
For hyperbolic partial differential equations, this is accomplished by using the spatial discretization that is stable for 
the downwinded analog of the operator.
This approach is similar to the one employed  in the classical SSP literature, where negative coefficients $\bij$ may be allowed if the 
corresponding operator is replaced by a downwinded operator \cite{gottliebshu1998,gottlieb2001, ketcheson2011}.
\end{rmk}

\bigskip

\noindent{\bf Example:} To demonstrate the practical importance of this thm, consider the partial differential equation
 \begin{align} 
u_t + 10 u_x +  \left( \frac{1}{2} u^2 \right)_x & = 0 \hspace{.75in}
    u(0,x)  =
\begin{cases}
1, & \text{if } 0 \leq x \leq 1/2 \\
0, & \text{if } x>1/2 \nonumber
\end{cases}
\end{align}
on the domain $[0,1]$ with periodic boundary conditions. We discretize the spatial grid with $400$ points and use
a first-order upwind difference to semi-discretize the linear term $L u \approx -10 u_x$,  and a fifth order WENO finite difference for the  
nonlinear terms $N(u) \approx - \left( \frac{1}{2} u^2 \right)_x $.  
For the time discretization, we use the integrating factor method based on the 
explicit eSSPRK(3,3) Shu-Osher method \eqref{SSPRK33}:
\begin{eqnarray} \label{SOIF}
     u^{(1)} &= & e^{L\dt}u^n + e^{L\dt} \dt N(u^n) \nonumber \\
     u^{(2)} &= & \frac{3}{4} e^{\frac{1}{2}L\dt} u^n + \frac{1}{4} e^{-\frac{1}{2}L\dt} u^{(1)} + \frac{1}{4} e^{-\frac{1}{2}L\dt} \dt N(u^{(1)})  \nonumber \\
     u^{n+1} & = & \frac{1}{3}e^{L\dt}  u^n +\frac{2}{3} e^{\frac{1}{2}L\dt}  u^{(2)} +  \frac{2}{3} e^{\frac{1}{2}L\dt} \dt N(u^{(2)}). 
\end{eqnarray}
The appearance of exponentials with negative exponents is due to the fact that \eqref{SSPRK33} has decreasing abscissas.
For comparison we also use a IFRK(3,3) method based on an explicit SSP Runge--Kutta method with non-decreasing abscissas,
denoted eSSPRK$^+(3,3)$ (which will be presented in \eqref{eSSPRK33+})
\begin{eqnarray}     \label{SSPIFRK33}
     u^{(1)} &= & \frac{1}{2}e^{\frac{2}{3} \dt L} u^n +  \frac{1}{2} e^{\frac{2}{3} \dt L}  \left(u^{n}+\frac{4}{3}\dt N(u^n)\right) \nonumber \\
     u^{(2)} &= & \frac{2}{3}  e^{\frac{2}{3} \dt L} u^n + \frac{1}{3} \left(u^{(1)} + \frac{4}{3} \dt N(u^{(1)})\right) \nonumber \\
     u^{n+1} & = & \frac{59}{128} e^{ \dt L} u^n + \frac{15}{128} e^{ \dt L}  \left(u^n +  \frac{4}{3} \dt N(u^n)\right)  \\
     && + \frac{27}{64} e^{\frac{1}{3} \dt L} \left(u^{(2)} 			+ \frac{4}{3} \dt N(u^{(2)})\right). \nonumber
\end{eqnarray}
The eSSPRK$^+(3,3)$ method this integrating factor is based on has SSP coefficient $\sspcoef = \frac{3}{4}$, 
which is smaller than the $\sspcoef=1$ of the Shu-Osher method \eqref{SSPRK33},
due to the restriction on the non-decreasing abscissas. thm \ref{thm:SSPIF} above tells us that the IFRK
method \eqref{SSPIFRK33}
will be SSP while the IFRK method  \eqref{SOIF} based on the Shu-Osher method \eqref{SSPRK33}  will not be.

We  selected different values of $\dt$ and used each one to evolve the solution  25 time steps using 
the IFRK methods \eqref{SOIF} and \eqref{SSPIFRK33}.
We  calculated the maximal rise in total variation over each stage for  25 time steps. 
In Figure \ref{fig:motivating} we show the $log_{10}$ of the maximal rise in total variation vs. the value of $\lambda = \frac{\dt}{\dx}$ of the evolution 
using  \eqref{SOIF} (in blue) and  using \eqref{SSPIFRK33} (in red). We observe that the results from 
method  \eqref{SOIF} have a large maximal rise in total variation even for very small values of $\lambda$,
while the results from \eqref{SSPIFRK33}  maintain a small maximal rise in total variation up to $\lambda \approx 0.8$.

\begin{figure}
\centering
 \label{fig:motivating} \includegraphics[scale=.375]{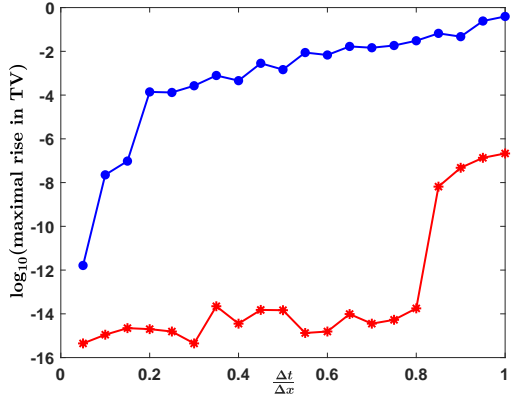} 
    \caption{ Total variation behavior of the evolution over 25 time-steps using the 
     integrating factor methods 
    \eqref{SOIF} (blue dots) and \eqref{SSPIFRK33} (red stars).
}
\end{figure}

This example clearly illustrates that basing an IFRK method on an explicit SSP Runge--Kutta method 
is not enough to ensure the preservation of a strong stability property. In this case, we {\em must} use the non-decreasing abscissa 
condition in thm \eqref{thm:SSPIF} to ensure that the strong stability property is preserved.

\section{Formulating the optimization problem} \label{sec:optimization}

Our aim is to find eSSPRK(s,p) methods to evolve an equation of the form \eqref{ode}
which have non-decreasing abscissas and 
 largest SSP coefficient $\sspcoef$. 
We denote these methods eSSPRK$^+$(s,p).
These methods can then be used to produce an integrating factor method \eqref{rkIFSO}
that has a guarantee of nonlinear stability,
as we showed in Section \ref{sec:SSPIF} above.  
Following the approach developed  by Ketcheson \cite{ketcheson2008},
we formulate an optimization problem
 similar to the one used for explicit SSP Runge-Kutta methods \cite{SSPbook2011} but with one additional constraint.

Although the SSP coefficient $\sspcoef$ is most easily seen in Shu-Osher form, constructing the optimization problem is easier when the method is written in Butcher form:
\begin{eqnarray} \label{butcher}
u^{(i)} & = & u^{n} + \dt \sum_{j=1}^{i-1} a_{ij} F(u^{(j)}) \; \; \; \; (1 \leq i \leq s) 
\\ \nonumber
u^{n+1} & = & u^n + \dt \sum_{j=1}^s b_j  F(u^{(j)}).
\end{eqnarray}
We can put all the $a_{ij}$ values into a matrix $\mA$ and all the $b_{j}$ into a vector $\vb$. Then we define the vector of abscissas $\vc = \mA \ve$, where 
$\ve$ is the vector of ones of the appropriate length.
We rewrite \eqref{butcher}  in vector form
\[Y = \ve u^n + \Delta t \mS  F(Y) \] where  $\mS$ is the square matrix defined by:
 \[ \mS =  \left( \begin{array}{ll} \mA & 0 \\ \vb^T & 0  \\ \end{array} \right) . \]
We add   $r \mS Y$ to each side to obtain
\begin{eqnarray*}
\left( I +r \mS   \right)  Y = 
\ve u^n  + r \mS \left( Y + \frac{\Delta t}{r}  F(Y) \right)   \quad   \Rightarrow  \quad
 Y  =  \mR (\ve u^n) + \mP \left( Y + \frac{\Delta t}{r}  F(Y) \right) ,
\end{eqnarray*}
for  $\mR  =  \left( I +r \mS   \right)^{-1}$ and  $\mP  =  r \mR \mS . $
Clearly, if $\mR \ve$ and $\mP$  have all non-negative components, 
we have a  convex combination of forward Euler steps. Therefore, 
 the strong stability property \eqref{monotonicity} will be preserved
under the modified time-step restriction $\dt \leq  r \DtFE  $.

As discussed in \cite{SSPbook2011}, the goal is to maximize the value of $r$ subject to the constraints
\begin{subequations} \label{optimization}
\begin{align}
 \left( I +r \mS  \right)^{-1} \ve   \geq 0    \label{eq:constraints1}  \\
r  \left( I +r \mS  \right)^{-1} \mS \geq 0   \label{eq:constraints2}  \\
 \tau_k(\mA, \vb) = 0 \; \; \; \mbox{for} \; \; \;  k=1, . . ., P,   \label{eq:constraints3} 
 \end{align}
 Where in   the inequalities are all component wise and $\tau_k$ in \eqref{eq:constraints3} are the order conditions.
 
In addition to the constraints \eqref{eq:constraints1} --  \eqref{eq:constraints3}, we must also add  the condition 
that  the abscissas are non-decreasing
\begin{equation} \label{eq:constraint4}
c_1 \leq c_2 \leq \dots \leq c_s \leq 1.
\end{equation}
\end{subequations}
Solving this optimization problem will generate an  explicit SSP  Runge--Kutta method
with coefficients $\mA$ and $\vb$  such that the abscissas are non-decreasing, with a SSP coefficient $\sspcoef = r$.

\subsection{Order conditions} \label{orderconditions}
The equality constraints \eqref{eq:constraints3} for the optimization problem above come from the order conditions, 
which were derived in \cite{Butcher}.
Below are the order conditions for methods up to fourth order. For first order a method must satisfy the consistency 
condition:
\begin{align*}
\vb^T \ve & = 1.
\end{align*}
In addition to this condition second order methods must also satisfy:
\begin{align*} 
\vb^T \vc & =  \frac{1}{2}.
\end{align*}
There are two more order conditions required to obtain third order:
\begin{align*} 
\vb^T \left(\vc \cdot \vc \right) & = \frac{1}{3}, &
\vb^T\mA\vc=\frac{1}{6}.
\end{align*}
For fourth order four additional conditions must be satisfied:
\begin{align*} 
\vb^T \left(\vc \cdot \vc \cdot \vc \right)  = \frac{1}{4}, \; \; \;
\vb^T \left(\vc \cdot \mA \vc \right) = \frac{1}{8}, \; \; \;  
\vb^T\mA \left(\vc \cdot \vc\right)  = \frac{1}{12}, \; \; \; 
\vb^T\mA^2 \vc  = \frac{1}{24}.\; \; \; 
\end{align*}
Note that $(\va \cdot \vb)$ denotes element-wise multiplication.
We do not present the order conditions past fourth order since there are no explicit SSP Runge--Kutta methods greater than fourth order.

\section{Optimal and optimized methods\label{sec:optimal}}
The optimization problem above was implemented in {\sc Matlab} (as in \cite{ketcheson2008, ketcheson2009a, ketcheson2009}), 
and used to find optimized eSSPRK$^+$ methods of up to ten stages and fourth order.  These methods have non-decreasing abscissas and so
can be used as a basis for explicit SSP integrating factor Runge--Kutta (eSSPIFRK) methods.

\begin{figure}[t]
\includegraphics[scale=.3275]{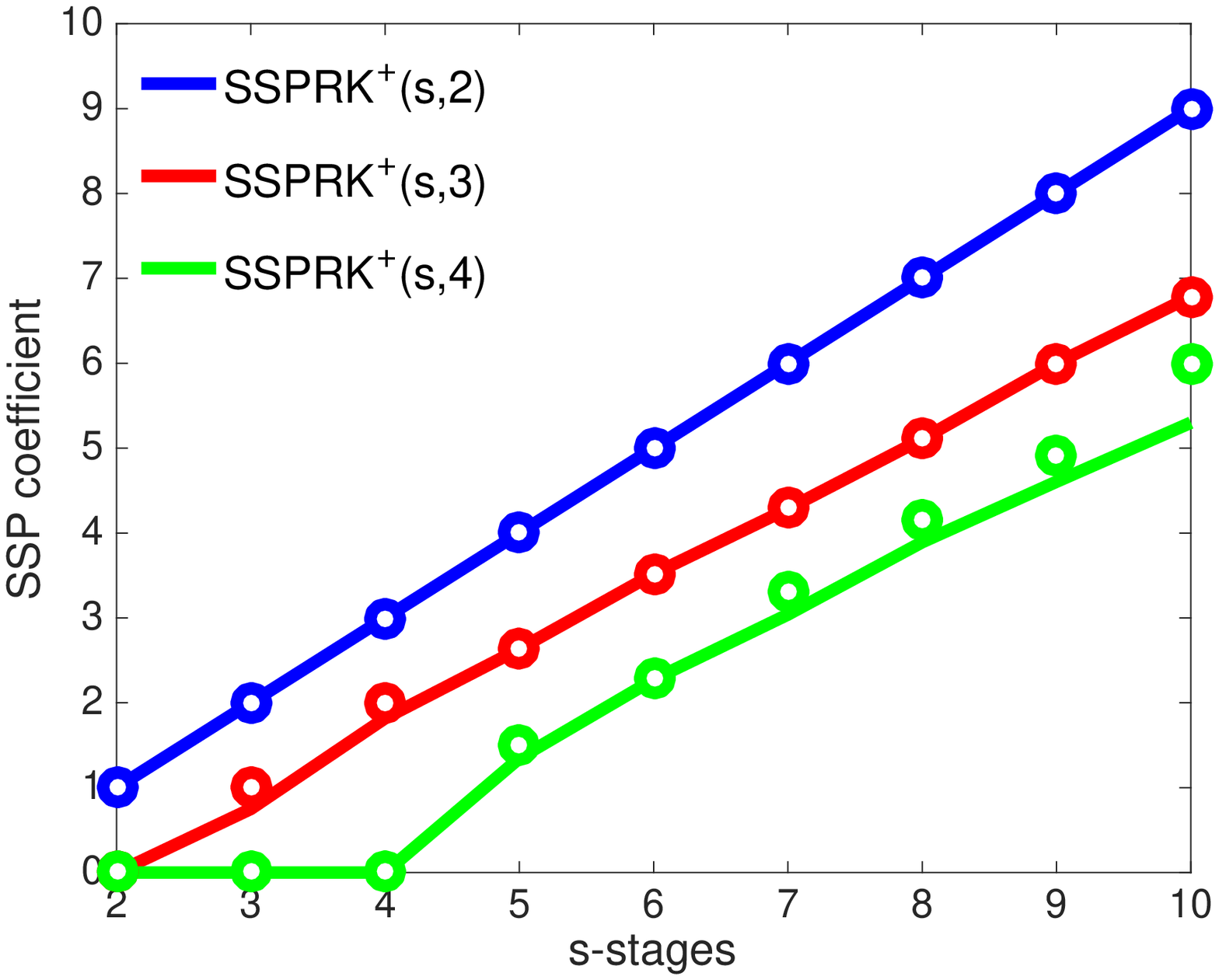} 
\includegraphics[scale=.3275]{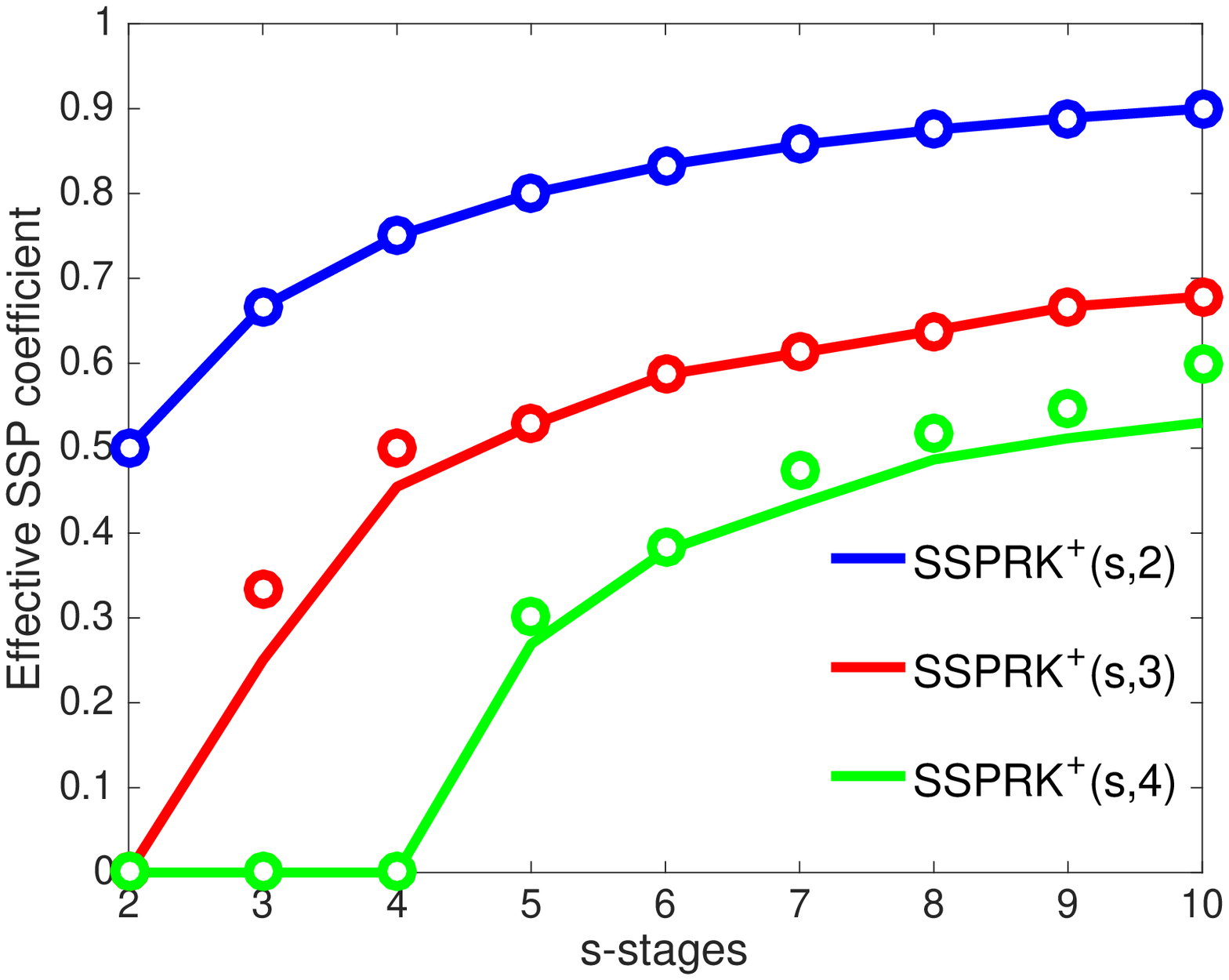}\\ \vspace{-.25in}
\caption{SSP coefficient (left) and effective SSP coefficient (right) for orders 
 for $p=2$ (blue), $p=3$ (red) and $p=4$ (green) methods.
 The circles indicate the SSP coefficient of the optimized  eSSPRK methods
 while the lines are the SSP coefficients of the optimized  eSSPRK$^+$ methods.}
 \label{fig:SSPcoeff}
\end{figure}

  The SSP coefficients and effective SSP coefficients of the optimized eSSPRK$^+$ methods
  are listed in Tables \ref{tab:SSPcoefIF} and \ref{tab:effSSPcoefIF}.
 The SSP coefficients of this family of methods are  compared to those of the optimized  eSSPRK methods with no constraint
 on the abscissas in Figure
 \ref{fig:SSPcoeff}, where the circles indicate the SSP coefficient of the optimized  explicit SSPRK methods 
  while the lines are the SSP coefficients of the optimized  explicit SSPRK$^+$ methods.

We observe that the optimal second order methods we found have the same SSP coefficients
as the previously known SSP Runge--Kutta methods. This  is not surprising as the abscissas of those optimal methods 
are non-decreasing, so our optimization routine found the previously known optimal methods. These eSSPRK$^+$(s,2) 
methods have SSP coefficient $\sspcoef=s-1$ and  effective SSP coefficient $\ceff=\frac{s-1}{s}$.
 In Section \ref{sec:order2} we give the coefficients for  these methods in both Shu-Osher form and 
 Butcher form and show that the abscissas are indeed non-decreasing.

 In the third order case, the additional requirement that the abscissas be non-decreasing results in smaller SSP coefficients
 than the typical explicit  SSP Runge--Kutta methods. For example, the optimal eSSPRK(3,3) Shu-Osher method 
 \eqref{SSPRK33} has SSP coefficient $\sspcoef = 1$ while the optimal   eSSPRK$^+$(3,3)  method 
  has  SSP coefficient $\sspcoef= \frac{3}{4}$, as we will prove in Section  \ref{sec:order3}. This loss in the 
 SSP coefficient is also evident for the eSSPRK$^+$(4,3) method  ($\sspcoef=\frac{20}{11}$)
 compared to the eSSPRK(4,3) method ($\sspcoef=2$). However, as we add more stages the 
impact of the additional requirement of non-decreasing abscissas becomes negligible and the SSP coefficients
of the eSSPRK$^+$(s,3) methods are very close to  those of the standard eSSPRK(s,3) methods,
as seen in Figure  \ref{fig:SSPcoeff}.
  
  \begin{table}[t!]
\centering
\setlength\tabcolsep{4pt}
\begin{minipage}{0.48\textwidth}
\centering
\begin{tabular}{|c|lll|} \hline
  \diagbox{s}{p}  &  2 & 3 & 4 \\
 \hline
    1  &  -           &  -           &   -\\   
    2  &  1.0000 &  -           &   -\\    
    3  &  2.0000 &  0.7500 &   -\\
    4  &  3.0000 &  1.8182 &   -\\
5  &  4.0000 &  2.6351 &   1.3466\\
    6  &  5.0000 &  3.5184 &   2.2738\\
    7  &  6.0000 &  4.2857 &   3.0404\\
    8  &  7.0000 &  5.1071 &   3.8926\\
   9  &  8.0000 &  6.0000 &   4.6048\\
   10 &  9.0000 &  6.7853 &   5.2997\\
\hline
\end{tabular}
\caption{SSP coefficients of the optimized eSSPRK$^+$(s,p) methods.}
\label{tab:SSPcoefIF} 
\end{minipage}%
\hfill
\begin{minipage}{0.48\textwidth}
\centering
\begin{tabular}{|c|lll|} \hline
  \diagbox{s}{p}  &  2 & 3 & 4 \\
 \hline
    1  &   -           &   -           &   -\\   
    2  &   0.5000 &   -           &   -\\
    3  &   0.6667 &   0.2500 &   -\\
    4  &   0.7500 &   0.4545 &   -\\
    5  &   0.8000 &   0.5270 &   0.2693\\
  6  &   0.8333 &   0.5864 &   0.3790\\
    7  &   0.8571 &   0.6122 &   0.4343\\
    8  &   0.8750 &   0.6384 &   0.4866\\
    9  &   0.8889 &   0.6667 &   0.5116\\
   10 &   0.9000 &   0.6785 &   0.5300\\
\hline
\end{tabular}
\caption{Effective SSP coefficients of the optimized eSSPRK$^+$(s,p) methods..}
 \label{tab:effSSPcoefIF} 
\end{minipage}
\end{table}

In the fourth order case, the SSP coefficient of the optimized eSSPRK$^+$ 
are certainly smaller than those of the corresponding eSSPRK methods. 
In fact, this does not significantly improve as we increase the 
number of stages. Notably, the optimal eSSPRK(10,4) method  found by 
Ketcheson \cite{ketcheson2008} has an SSP coefficient of $\sspcoef=6$ while the corresponding 
eSSPRK$^+$(10,4) method
has SSP coefficient $\sspcoef=5.3$, a reduction of over 10\%. 
An exception to this is the  optimized eSSPRK$^+$(6,4) method 
in which the non-descreasing abscissa requirement results in only a $1\%$ reduction of the SSP coefficient
compared to the eSSPRK(6,4).


\subsection{Sub-optimal explicit SSP  Kinnmark and Gray Runge--Kutta methods with non-decreasing abscissas} \label{sec:KGmethods}
In \cite{KG1984}, Kinnmark and Gray presented a set of Runge--Kutta methods for linear problems. More precisely, they presented the linear
stability polynomials for these methods. These methods were designed for use with problems that require a linear stability polynomials that include a
large area of the imaginary axis. It is interesting to investigate what types of SSP coefficients the methods described in \cite{KG1984} can have. 
To do so, we modified our code that finds SSP Runge--Kutta methods of $s$ stages and order $p$
 to include the linear stability polynomials of Kinnmark and Gray and used this code to find optimized SSP methods 
with $(s,p)= (3,3), (5,3), (6,4)$. 
We call these eSSPKG(s,p) methods and present their SSP coefficients and a comparison to the corresponding
eSSPRK(s,p) methods in Table  \ref{KGSSPcoef}. 
All explicit SSP Runge--Kutta methods of $s=p=3$ have the same stability polynomial, so that the 
Kinnmark and Gray methods eSSPKG(3,3) have the same linear stability region as eSSPRK(3,3).
Although the SSP coefficients of the other SSP Kinnmark and Gray methods are smaller than those 
of the typical SSP Runge--Kutta methods, they may be used if the linear stability regions
of Kinnmark and Gray are of interest. Note that the approach of optimizing an SSP method for a given linear stability region was originally done by \cite{KubatkoKetcheson}.  For comparison, we provide the linear stability regions in Figure \ref{KGstability}. We observed that, as expected,
the Kinnmark Gray methods have larger imaginary axis stability, but smaller overall regions.

Next, we added the requirement that the abscissas are non-decreasing to find optimized SSPKG+(s,p) methods
for $(s,p)= (3,3), (5,3), (6,4)$, and compare their SSP coefficients to those of the SSPRK+(s,p), also in Table \ref{KGSSPcoef}.  
These methods are suitable for use
with Lawson-type integrating factor methods. In Figure \ref{KGstability} we show the linear stability regions of the 
$(s,p)=(5,3)$ and $(s,p)=(6,4)$ methods, as well.
We test the SSPKG methods as well as their use within the integrating factor approach in Example 3,
where we see that their  allowable time-step for preserving the TVD properties of a simple benchmark method  are generally smaller than those
of the SSPRK+ methods and the SSPRKIF methods.

\begin{figure}[t]
\centering
  \includegraphics[scale=.3]{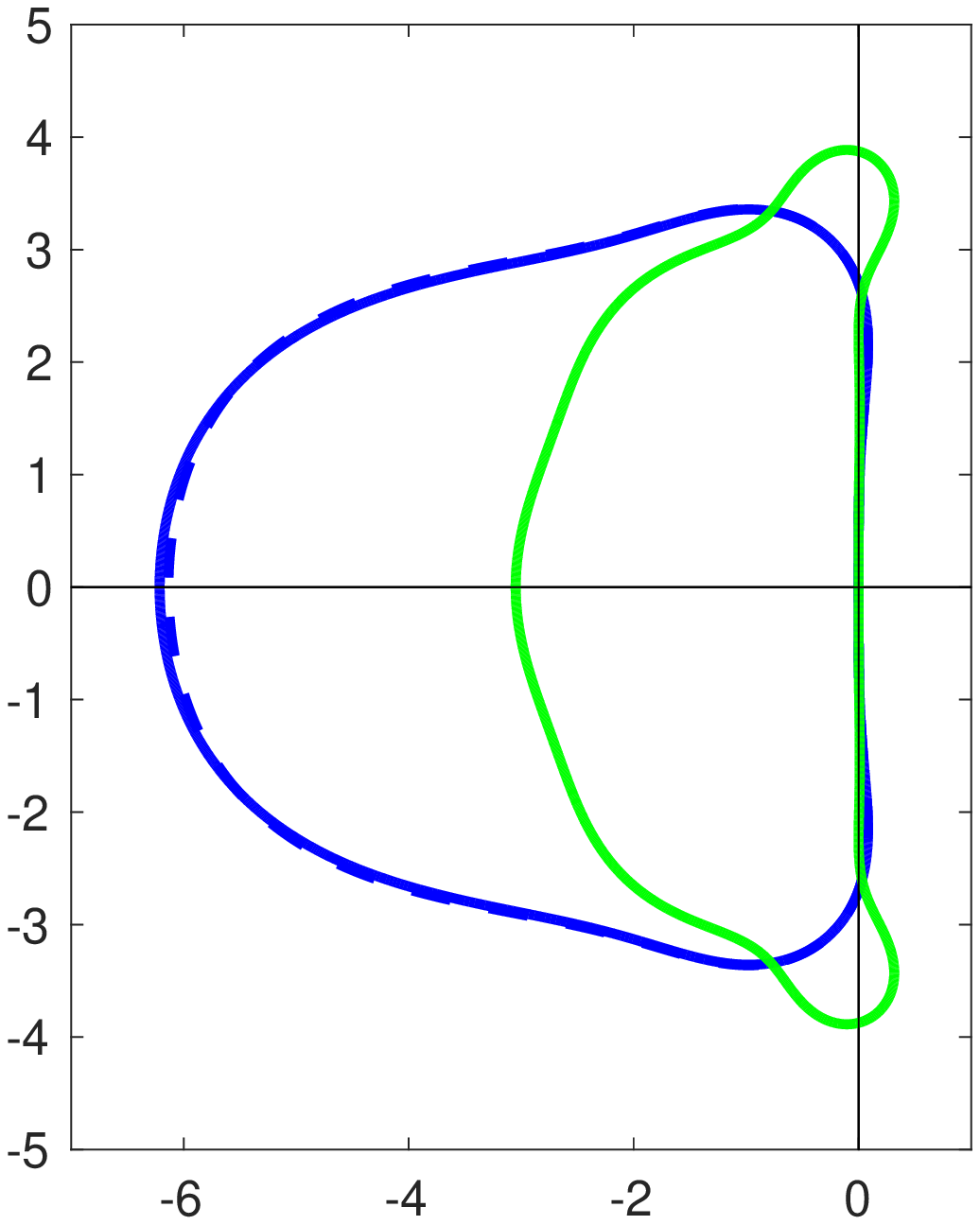}  \hspace{-.7in}
  \includegraphics[scale=.3]{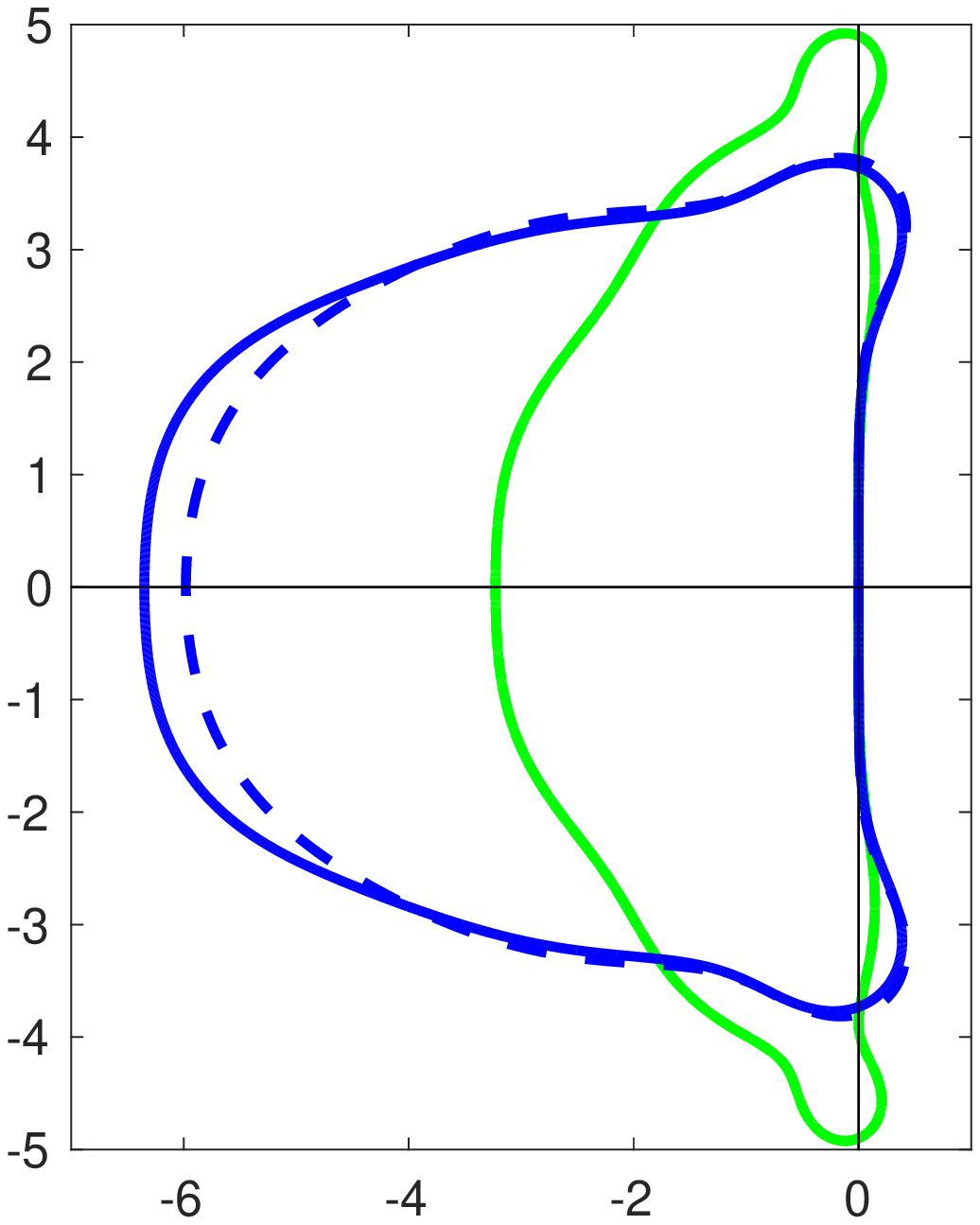} 
    \caption{ Linear stability region of the Kinnmark and Gray methods (in green) compared to the eSSPRK method (blue solid line)
    and the eSSPRK+ method (blue dashed line). Left: method with $(s,p)=(5,3)$; Right: method with $(s,p)=(6,4)$.
     \label{KGstability}}
\end{figure}

\bigskip

\begin{table}[h]
\begin{center}
\begin{tabular}{|ll|ll|} \hline
method & $\sspcoef$ & method & $\sspcoef$ \\ \hline
eSSPRK(5,3) &   2.6506  		&   eSSPRK+(5,3) & 2.6351 \\
{\bf eSSPKG(5,3)} &    1.0000 		&   {\bf  eSSPKG+(5,3)}  & 0.8750 \\
eSSPRK(6,4) &    2.2945		&   eSSPRK+(6,4) & 2.2738 \\
{\bf eSSPKG(6,4) } &    0.9904  	&   {\bf eSSPKG+(6,4) }&0.7851  \\ \hline
\end{tabular}
\caption{SSP coefficients of the optimized SSP Runge--Kutta methods designed for the Kinnmark and Gray stability regions
(eSSPKG methods, in bold) compared to typical SSP Runge--Kutta methods (eSSPRK methods), 
with (right column) and without (left column) non-decending abscissas.} \label{KGSSPcoef}
\end{center}
\end{table}

\newpage

\subsection{Optimal second order explicit SSP  Runge--Kutta methods with non-decreasing abscissas} \label{sec:order2}
We mentioned above that the optimal  eSSPRK(s,2) methods have non-decreasing coefficients.
In this section we review  these methods, first presented in  \cite{gottlieb2001},
and show that the abscissas are in fact increasing. These methods
can be written in Shu-Osher form (where $u^{(0)}=u^{n}$) :

\begin{equation*}
\begin{split}
u^{(i)}&= u^{(i-1)} +\frac{\Delta t}{s-1}F (u^{(i-1)}), \quad i=1,\dots,s-1\\
u^{(s)}&=\frac{1}{s}u^{0}+\frac{s-1}{s}\left(
u^{(s-1)} +\frac{\Delta t}{s-1}F(u^{(s-1)}) \right) \\
u^{n+1}&=u^{(s)}.
\end{split}
\end{equation*}

%
In Butcher form, this becomes
\[
\mA
=
\begin{bmatrix}
    0 & 0 & 0 & \dots  & 0 & 0\\
    \frac{1}{s-1} & 0 & 0 & \dots  & 0 & 0\\
    \frac{1}{s-1} &  \frac{1}{s-1} & 0 & \dots  & 0 & 0\\
    \vdots & \ddots & \ddots & \ddots & \vdots &\vdots\\
    \frac{1}{s-1} & \frac{1}{s-1} & \frac{1}{s-1} &  \frac{1}{s-1}  & 0 & 0 \\
    \frac{1}{s-1} & \frac{1}{s-1} & \frac{1}{s-1} & \frac{1}{s-1} &\frac{1}{s-1} & 0
\end{bmatrix} , \; \; 
\vb^T
=
\begin{bmatrix}
    \frac{1}{s} \\ \frac{1}{s} \\ \frac{1}{s} \\ \vdots  \\ \frac{1}{s} \\ \frac{1}{s}\\
\end{bmatrix} \; ,\]
with  abscissas
\[ \vc^T=  \left[0,  \frac{1}{s-1},   \frac{2}{s-1}, \dots  \frac{s-2}{s-1},   1 \right] .\]
 Clearly, these optimal explicit SSP Runge--Kutta methods have increasing abscissas, and are therefore
suitable use with an integrating factor approach to create eSSPIFRK methods.
 

\subsection{Optimized third order explicit SSP  Runge--Kutta methods with non-decreasing abscissas}  \label{sec:order3}

\begin{thm}
The eSSPRK$^+$(3,3) method given by 
\begin{eqnarray} \label{eSSPRK33+}
     u^{(1)} &= & \frac{1}{2} u^n+\frac{1}{2}\left(u^n +  \frac{4}{3}\dt F(u^n)\right) \nonumber \\
     u^{(2)} &= & \frac{2}{3} u^n + \frac{1}{3} \left( u^{(1)} + \frac{4}{3} \dt F(u^{(1)}) \right)  \nonumber \\
     u^{n+1} & = & \frac{59}{128} u^n + \frac{15}{128} \left(  u^{(1)} +  \frac{4}{3} \dt F(u^{(1)}) \right)+ \frac{27}{64} 
         \left(u^{(2)} + \frac{4}{3} \dt F(u^{(2)}) \right).
\end{eqnarray}
is strong stability preserving with SSP coefficient $\sspcoef=\frac{3}{4}$ and is optimal among all 
eSSPRK$^+$(3,3) methods.
\end{thm}

\begin{proof}
This method is given in its canonical Shu-Osher form. Clearly, we have a convex combination of forward Euler steps with
time-step $ \frac{4}{3} \dt $ and so this method is SSP with $\sspcoef=\frac{3}{4}$.

To show that this is optimal among all possible eSSPRK$^+$(3,3) methods, we  follow along the lines
of the proof in \cite{SSPbook2011}. We assume that $\sspcoef>\frac{3}{4}$, which means that
$ \frac{\alpha_{ij}}{\beta_{ij}}>\frac{3}{4}$ or  
\[ \alpha_{ij}>\frac{3}{4}\beta_{ij} \quad \mbox{for any} \quad  i,j, \]
and proceed with a proof by contradiction.

First, recall that we can transform between the Shu-Osher coefficients $\alpha,\beta$ and the Butcher array coefficients
$\mA$ and $\vb$ as follows:
\begin{equation*}
\begin{split}
a_{21}=& \beta_{10}, \; \; 
a_{31}= \beta_{20} + \alpha_{21}\beta_{10}, \; \; 
a_{32}= \beta_{21}\\
b_{1}=& \alpha_{32}\alpha_{21}\beta_{10} + \alpha_{31}\beta_{10} + \alpha_{32}\beta_{20} + \beta_{30}, \; \; 
b_{2}= \alpha_{32}\beta_{21} + \beta_{31}, \; \; 
b_{3}= \beta_{32}\\
c_1=&0, \; \; c_2 = a_{21}, \; \;  c_3 = a_{31}+ a_{32} .\\
\end{split}
\end{equation*}

Note that for the method to be SSP, the  coefficients must all be non-negative.
If the abscissas are non-decreasing, there are two possible cases: $c_2 = c_3 $ and $c_2 < c_3$.
We consider each of these cases separately. 

\noindent [{\bf Case (a)}] If the abscissas are equal, they must be $c_2 = c_3 = \frac{2}{3}$. 
The coefficients in this case satisfy
\[ a_{21} = \frac{2}{3},  \; \; a_{31} = \frac{2}{3} - \frac{1}{4 \omega} ,  \; \; a_{32} = \frac{1}{4 \omega}, \; \; 
b_1= \frac{1}{4} ,  \; \;  b_2= \frac{3}{4} - \omega ,  \; \;  b_3 = \omega  \]
for parameter $\omega$.
The assumption that $ \alpha_{ij}>\frac{3}{4}\beta_{ij} $ and the  non-negativity assumption on the coefficients results in
\[ b_{2}= \alpha_{32}\beta_{21} + \beta_{31} \geq \alpha_{32}\beta_{21} > \frac{3}{4}  \beta_{32}\beta_{21}  
= \frac{3}{4}  b_{3} a_{32} =  \frac{3}{16} 
\implies \omega < \frac{9}{16} . \]

On the other hand, 
\[ a_{31} = \beta_{20} + \alpha_{21}\beta_{10} \geq  \alpha_{21}\beta_{10} >  \frac{3}{4} \beta_{21} \beta_{10}  = \frac{3}{4} a_{32} a_{21} =
 \frac{3}{4} \frac{1}{4 \omega}  \frac{2}{3}= \frac{1}{8 \omega}  \]
so that 
\[ \frac{2}{3} - \frac{1}{4 \omega}  > \frac{1}{8\omega} 
\implies \omega > \frac{9}{16} ,\]
which contradicts the bound on $\omega$ above.

\noindent [{\bf Case (b)}] If the two abscissas are not equal, and we require non-decreasing abscissas, we must have
$ c_2 < c_3$.  In this case, the coefficients are given by a two parameter system, where the parameters are the abscissas $c_2$ and $c_3$.
\begin{eqnarray*} 
a_{21}= c_{2},  \; \; a_{31}= \frac{3c_{2}c_{3}(1-c_{2})-c_{3}^2}{c_{2}(2-3c_{2})} , \; \;
a_{32}= \frac{c_{3}(c_{3}-c_{2})}{c_{2}(2-3c_{2})} \\
 b_{1}= 1 + \frac{2-3(c_{2}+c_{3})}{6c_{2}c_{3}},
\; \; b_{2}= \frac{3c_{3}-2}{6c_{2}(c_{3}-c_{2})},
 \; \;  b_{3}= \frac{2-3c_{2}}{6c_{3}(c_{3}-c_{2})}.\\
\end{eqnarray*}

The requirement that 
$ c_2 < c_3$ and that both $b_2$ and $b_3$ are non-negative gives 
\begin{equation*}
\begin{split}
b_{2}=& \frac{3c_{3}-2}{6c_{2}(c_{3}-c_{2})}\geq0 \implies c_{3} \geq \frac{2}{3}\\
b_{3}=& \frac{2-3c_{2}}{6c_{3}(c_{3}-c_{2})}\geq0 \implies c_{2} \leq \frac{2}{3}.\\
\end{split}
\end{equation*}

Now begin with $a_{31}= \beta_{20} + \alpha_{21}\beta_{10} $, and recall that  $\alpha_{21}>\frac{3}{4}\beta_{21}$ and 
$\beta_{10}=a_{21}=c_{2}$ and $\beta_{21} = a_{32}$, so that
\[ a_{31} >  \frac{3}{4}\beta_{21}c_{2} = \frac{3}{4} a_{32} c_{2}
 \implies c_3-a_{32}> \frac{3}{4}a_{32}c_{2}
  \implies c_3> a_{32}\left(1+\frac{3}{4}c_{2}\right) \]
  which requires $ a_{32} < \frac{c_3}{\left(1+\frac{3}{4}c_{2}\right)}$.
  
  Using the definition of $a_{32}$, this means
\[ \frac{c_{3}(c_{3}-c_{2})}{c_{2}(2-3c_{2})}  <  \frac{c_3}{\left(1+\frac{3}{4}c_{2}\right)} 
\implies c_{3}< \frac{3c_{2}-\frac{9}{4}c_{2}^2}{\left(1+\frac{3}{4}c_{2}\right)} .\]

Next, we look use the fact that $\beta_{31}\geq0$ to obtain
\[ b_2 = \alpha_{32}\beta_{21} + \beta_{31} \geq  \alpha_{32}\beta_{21}  > \frac{3}{4} b_3 a_{32} 
\implies a_{32}  < \frac{4}{3} \frac{b_2}{b_3} = \frac{4}{3}\frac{c_3(3 c_{3}-2)}{c_2(2-3 c_{2})}.
\]
Now we use $a_{32}= \frac{c_{3}(c_{3}-c_{2})}{c_{2}(2-3c_{2})}$ to conclude that
\[  \frac{c_{3}(c_{3}-c_{2})}{c_{2}(2-3c_{2})}  <   \frac{4}{3}\frac{c_3(3c_{3}-2)}{c_2(2-3 c_{2})} 
 \implies   c_3  >   -\frac{1}{3} c_{2}+\frac{8}{9} .\]

We now have two statements that need to be simultaneously true
\[
c_{3}  <  \frac{3c_{2}-\frac{9}{4}c_{2}^2}{\left(1+\frac{3}{4}c_{2}\right)}  \; \; \mbox{and} \; \; 
c_3   >   -\frac{1}{3} c_{2}+\frac{8}{9} \]
which means that we must have
\[ - \frac{1}{3}c_{2}+\frac{8}{9} <\frac{3 c_{2}-\frac{9}{4} c_{2}^2}{\left(1+\frac{3}{4} c_{2}\right)}
\implies
(3 c_{2}-2)^2 <0 \]
however this is a contradiction because $(3 c_{2}-2)^2$ is always greater than or equal to zero. This means that our original assumption
was not correct, and that if $ c_2 < c_3$ we cannot have $\sspcoef >\frac{3}{4}$.
\end{proof}

\subsection{Recommended SSP Runge--Kutta methods for use with integrating factor methods}
The optimal second order methods eSSPRK$^+$(s,2) listed above have sparse Shu-Osher representations and a general formula.
However, for the optimized third and fourth order methods, we do not have a general formula. 
In this section we list a few of the optimized
third and fourth order methods. 
The coefficients of all the methods we found can be downloaded as {\tt .mat} files from our github repository
\cite{SSPIFgithub}.

\bigskip

\noindent{\bf eSSPRK$^+$(4,3) }
This method has rational coefficients and sparse Shu-Osher matrices: 
{
\begin{eqnarray*}
u^{(1)}&=& u^{n} + \frac{11}{20}\Delta t F(u^{n})\\
u^{(2)}&=&\frac{3}{8} u^{n} 
	     +\frac{5}{8}\left(u^{(1)} + \frac{11}{20}\Delta tF(u^{(1)})\right)\\
u^{(3)}&=&\frac{4}{9}u^{n} 
	   +\frac{5}{9}\left(u^{(2)} + \frac{11}{20}\Delta tF(u^{(2)})\right)\\
u^{n+1}&=&\frac{111}{1331}u^{n}+ \frac{260}{1331}\left(u^{n}
+\frac{11}{20}\Delta t F(u^{n})\right)+\frac{960}{1331}\left(u^{(3)}+\frac{11}{20}\Delta tF(u^{(3)})\right)
\end{eqnarray*} }
The abscissas are $c_1=0, c_2= \frac{11}{20}, c_3=c_4 = \frac{11}{16}$.
This method has $\sspcoef=\frac{20}{11}$.

\smallskip

When many stages are required for a high order computation, the amount of storage, particularly for large simulations, may
become prohibitive. Low storage methods are of great interest in such cases.
Low storage Runge--Kutta methods were considered in \cite{Williamson1980,KennedyCarpenterLewis2000,vanderHouwen1972}.
More recently,  Ketcheson \cite{ketcheson2008} developed many low-storage SSP methods and showed that some of the
most efficient methods in terms of the SSP coefficient are also efficient in terms of storage.
This method is low-storage in the sense of \cite{ketcheson2008}, as many of the storage registers can be overwritten
during the implementation, assuming one is willing to recompute $F(u^{i})$ when needed \cite{ketcheson2008, SSPbook2011}. 
Due to the structure of the Shu-Osher matrices, only two memory registers are required for this method, 
rather than the full $s+1=5$ that would be needed for a naive implementation.
\smallskip

\noindent{\bf eSSPRK$^+$(9,3)}
This method has $\sspcoef = 6$  and features 
rational coefficients and sparse Shu-Osher matrices:
{
\begin{equation*}
\begin{split}
u^{(0)} =& u^n \\
u^{(i)}=&u^{(i-1)} + \frac{1}{6}\Delta tF(u^{(i-1)}) \; \; \mbox{for}\;  i=1, . . ., 4 \\
u^{(5)}=&\frac{1}{5}u^{n}
	     +\frac{4}{5}\left(u^{(4)} + \frac{1}{6}\Delta t F(u^{(4)})\right)\\
u^{(6)}=&\frac{1}{4}\left(u^{n} + \frac{1}{6}\Delta t F(u^{n})\right)
	   +\frac{3}{4}\left(u^{(5)} + \frac{1}{6}\Delta t F(u^{(5)})\right)\\
u^{(7)}=&\frac{1}{3} u^{(2)} 
	     +\frac{2}{3}\left(u^{(6)} + \frac{1}{6}\Delta t F(u^{(6)})\right)\\
u^{(8)}=&u^{(7)} + \frac{1}{6}\Delta tF(u^{(7)})\\
u^{n+1}=&u^{(8)}+\frac{1}{6}\Delta tF(u^{(8)}) .
\end{split}
\end{equation*}}
The abscissas are $c_1=0, c_2=\frac{1}{6}, c_3=\frac{2}{6}, c_4 = \frac{3}{6},
c_5= c_6 =c_7 = c_8 = \frac{4}{6}, c_9 =\frac{5}{6}$, which simplifies the computation of the matrix exponential, 
as only one needs to be computed.

This method  is also efficient in terms of memory as it may
be implemented with  only three storage registers (assuming one is willing to compute $F(u^{i})$ twice).
Although this is not a low-storage method in the sense of \cite{ketcheson2008}, i.e., it requires more than two registers,  
we do not require the full $s+1=10$  storage registers naively needed for implementing this method, but only three.

\smallskip

\noindent For fourth order methods, we no longer have rational coefficients.

\noindent {\bf eSSPRK$^+$(5,4)}
This method has SSP coefficient 
$\sspcoef=r=1.346586417284006$, and  non-decreasing abscissas 
$ c_1 =  0, \quad
c_2 \approx     0.4549,  \quad
c_3 = c_4  \approx   0.5165, \quad
c_5 \approx  0.9903 $:

{\small
\begin{eqnarray*}
u^{(1)}&=&0.387392167970373 \;u^{n} + 0.612607832029627 \left( u^{n} + \frac{\Delta t}{r} F(u^{n}) \right) \\
u^{(2)}&=&0.568702484115635 \;u^{n} + 0.431297515884365  \left( u^{(1)} +  \frac{\Delta t}{r}  \Delta tF(u^{(1)}) \right) \\
\end{eqnarray*}
\begin{eqnarray*}
u^{(3)}&=&0.589791736452092 \;u^{n} + 0.410208263547908 \left( u^{(2)} +   \frac{\Delta t}{r}  F(u^{(2)}) \right) \\
u^{(4)}&=&0.213474206786188 \;u^{n} + 0.786525793213812   \left( u^{(3)} +  \frac{\Delta t}{r}  F(u^{(3)}) \right) \\
u^{n+1}&=&0.270147144537063 \;u^{n} + 
0.029337521506634 \left( u^{n} + \frac{\Delta t}{r} F(u^{n}) \right)  \\
&+ & 0.239419175840559 \left( u^{(1)} +  \frac{\Delta t}{r}  \Delta tF(u^{(1)}) \right)
+ 0.227000995504038 \left( u^{(3)} +  \frac{\Delta t}{r}  F(u^{(3)}) \right) \\
&+&   0.234095162611706 \left( u^{(4)} +  \frac{\Delta t}{r}  F(u^{(4)}) \right) .
\end{eqnarray*}}

\bigskip

\noindent{\bf eSSPRK$^+$(6,4)}
This method has SSP coefficient 
$\sspcoef=r=2.273802749301517$, and non-decreasing abscissas
$ c_1 =  0, \quad 
c_2 \approx 0.4398, \quad  
c_3 \approx  0.4515, \quad 
c_4 = c_5 \approx 0.5461, \quad 
c_6 \approx. 0.9859  $:

{\small
\begin{eqnarray*}
u^{(1)}&=&    u^{n} + \frac{\Delta t}{r} F(u^{n})  \\
u^{(2)}&=&   0.486695314011133    u^{n}  +  0.513304685988867 \;  \left( u^{(1)} +  \frac{\Delta t}{r}  \Delta tF(u^{(1)}) \right)\\
u^{(3)}&=&   0.387273961537322 \;u^{n} + 0.612726038462678 \; \left( u^{(2)} +   \frac{\Delta t}{r}  F(u^{(2)}) \right)\\
u^{(4)}&=&   0.419340376206590  \;u^n+  0.048271190433595 \;  \left( u^{n} + \frac{\Delta t}{r} F(u^{n}) \right) +  \\
&&  0.532388433359815  \left( u^{(3)} +  \frac{\Delta t}{r}  F(u^{(3)}) \right)\\
u^{(5)}&=&   u^{(4)} +  \frac{\Delta t}{r}  F(u^{(4)}) \\                 
u^{n+1}&=&   0.122021674306995 \;u^{n} +
   0.104714614292281 \;  \left( u^{(1)} +  \frac{\Delta t}{r}  \Delta tF(u^{(1)}) \right)+ \\
&&   0.316675962670361 \;  \left( u^{(2)} +  \frac{\Delta t}{r}  \Delta tF(u^{(2)}) \right)+
   0.057551178672633 \;  \left( u^{(4)} +  \frac{\Delta t}{r}  \Delta tF(u^{(4)}) \right)+ \\
&&   0.399036570057730 \;  \left( u^{(5)} +  \frac{\Delta t}{r}  \Delta tF(u^{(5)}) \right).
\end{eqnarray*}}


\newpage

\section{Numerical Results\label{sec:test}}
In this section, we test the explicit SSP integrating factor  Runge--Kutta (eSSPIFRK) methods based on eSSPRK$^+$ methods 
 presented  in Section \ref{sec:optimal}  
for convergence and SSP properties. First, we test these methods for convergence on  a nonlinear system of ODEs  to confirm 
that the new methods exhibit the desired orders.  Next, we study the behavior of these methods in terms of their allowable 
time-step on  linear and nonlinear problems with spatial discretizations that  are provably total variation diminishing (TVD). 
%
%
 While the utility of SSP methods goes well beyond its initial purpose of preserving the TVD
properties of the spatial discretization coupled with forward Euler, the simple TVD test in this 
section  has been used extensively because it tends to demonstrate the sharpness of the 
SSP time-step.

\begin{rmk}
The cost of computation of the matrix exponential is a major factor that will determine the efficiency of these methods 
in practice. There are several approaches that can be taken here. The first and simplest approach is that the approximation
of the matrix exponential be done by evolving  $u' = Lu$ (where $u$ is a matrix and $u^0 = I$) 
numerically up to $t=\dt$ using an explicit SSP RK method with a sufficiently small stepsize
as in \cite{dubious} (a more recent approach combines this idea with a scale and square method \cite{AlMohyHigham}). 
This is not inefficient when performed only once per simulation, which is all that is required when $L$ 
is a constant coefficient operator. However, when storage is a consideration, and matrix-free approaches are desired,
there are a number of other efficient  approaches that have been proposed, and are under active consideration by 
several research groups  working on exponential time differencing methods.
Many of these methods produce the action of a matrix exponential at a cost less than the cost of an implicit solve.
The reader is referred to the  work of \cite{AlMohyHigham}, the EXPOKIT software \cite{Sidje}, 
the {\em phipm} adaptive method in \cite{NiesenWright}, 
and the KIOPS method of Tokman \cite{GaudreaultRainwaterTokman}.
Although this active area of research is outside the scope of our paper, it is of great interest as it
will be necessary bringing the SSPIFRK methods presented in this paper into practical use.
\end{rmk}

\subsection{Example 1: Convergence study}  
To verify the order of convergence of these methods  we test their performance on a nonlinear system of ODEs
\begin{eqnarray*}
u_1' &=& u_2 \\ 
u_2' &=&  (-u_1 + (1-u_1^2) u_2)
\end{eqnarray*}
known as the  van der Pol problem.
We split the problem in two different ways into a linear part $Lu$ and a nonlinear part $N(u)$ given by:
\[\mbox{(a)} \; \; \; L =  \left( \begin{array}{rr}
0 & 1 \\
-1 & 1\\
\end{array} \right), \; \; \; 
N(\vu) = \left(\begin{array}{c}
0 \\ -u_1^2 u_2
\end{array} \right)
 \]
 and
 \[\mbox{(b)} \; \; \; L =  \left( \begin{array}{rr}
0 & 1 \\
-1 & 0\\
\end{array} \right), \; \; \; 
N(\vu) = \left(\begin{array}{c}
0 \\ (1-u_1^2) u_2
\end{array} \right)
 \]
We use  initial conditions $\vu_0 = (2;  0)$, and  run  the problem to 
 final time $T_{final} =  0.50$, with $\Delta t = 0.02, 0.04, 0.06, 0.08, 0.10$.
 The exact solution (for error calculation)  was calculated by  MATLAB's 
ODE45 routine with tolerances set to {\tt AbsTol=}$10^{-15}$
and {\tt  RelTol=}$10^{-15}$.
For each splitting, we tested all the methods represented in Table \ref{tab:SSPcoefIF} above 
and calculated the slopes of the orders  by MATLAB's {\tt polyfit} function;
we found that they all exhibit the expected order of convergence.
Due to space constraints, we show only a representative selection 
 in Figure \ref{fig:VDPconv}. 
\begin{figure}[t]
\begin{center}
\includegraphics[scale=.325]{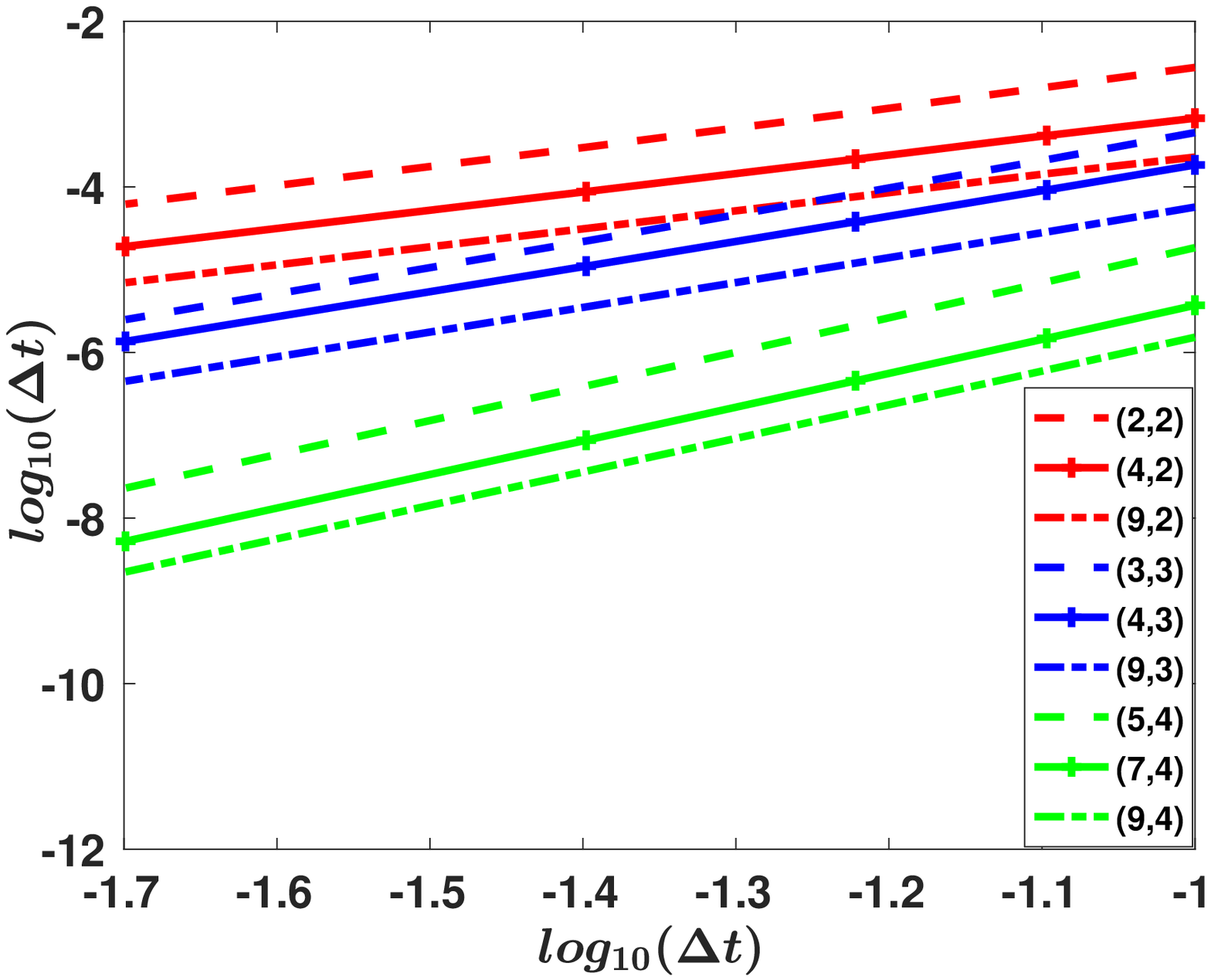}
\includegraphics[scale=.325]{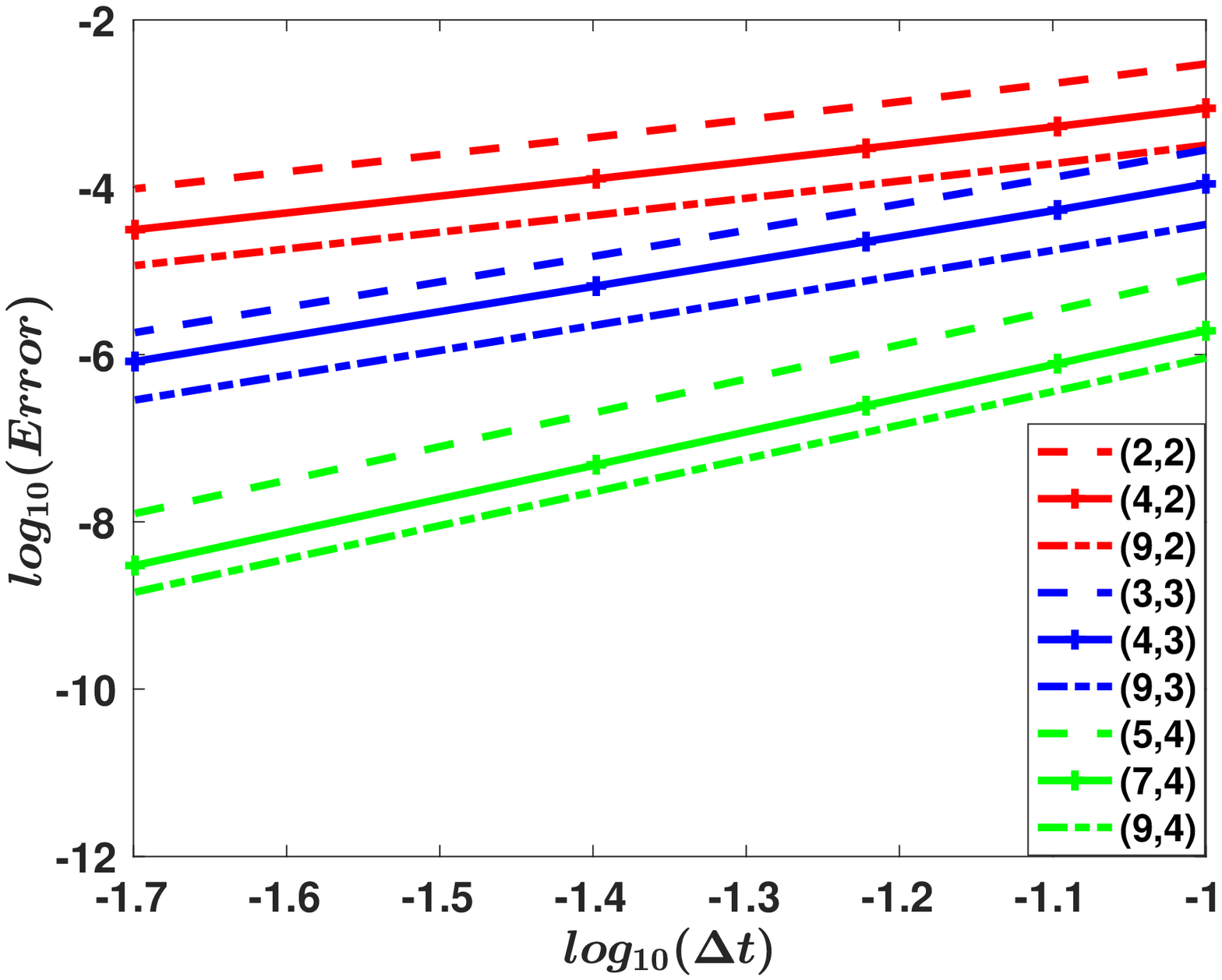}
\end{center} 
\caption{Example 1. 
The second order methods are in red, third order in blue, and fourth order in green.
The eSSPIFRK methods with $(s,p)=(2,2), (3,3), (5,4)$ have a dashed line,
those with $(s,p)=(4,2), (4,3), (7,4)$ have a solid line, and 
those with $s=9$ have a dot-dash line. On the left is splitting (a) while on the right is splitting (b).
} \label{fig:VDPconv}
\end{figure}  
While the splitting affects the magnitude of the errors, we see that the order of the errors is not
affected.
As expected, the error constants are smaller for methods with more stages.

Note that we used a van der Pol problem that is not highly oscillatory. This is because 
we wish to avoid the order reduction that is known to occur with 
integrating factor methods. This convergence study purposely avoids this issue in order to test the formal 
convergence of the generated methods.

\subsection{Example 2: Accuracy study}  
Consider the 
\begin{align}
u_t + 10 u_x +  \left( \frac{1}{2} u^2 \right)_x & = 0 \hspace{.75in}
    u(0,x)  = e^{\sin(2\pi x)}
\end{align}

on the domain $ 0 \leq x \leq 1$.
We use a first order upwind finite difference to spatially discretize the linear advection term
and the fifth order WENO for the nonlinear term. We use $64$  points in space and
the  {\tt globalorder.m} script in the package EXPINT  \cite{EXPINT,EXPINTpackage} with its
built-in exponential time-differencing (ETD) Runge--Kutta methods of orders $p=2,3,4$ 
(the schemes by Cox and Matthews called ETDRK2, ETDRK3, and ETDRK4 in EXPINT) 
and our eSSPIFRK methods with
$(s,p)= (2,2), (3,3), (5,4)$. The {\tt globalorder.m} script uses MATLAB's embedded 
ODE15s with $AbsTol=10^{-15}$ and $RelTol=5 \times 10^{-14}$
 to compute the highly accurate reference solution.
 In Figure \ref{fig:Accuracy} (left) we observe that the eSSPIFRK methods are competitive with 
 the ETD methods in terms of accuracy.   Despite the fact that we see order reduction
 in the third and fourth order eSSPIFRK methods, the accuracy of these methods is 
 comparable to that of the corresponding ETD method.

 A comparison of the CPU times needed for a given level of accuracy 
in  Figure \ref{fig:Accuracy} (right) 
  reveals that the ETD methods
 are generally somewhat  more efficient on this smooth problem.
 However,  we will see below that the ETD methods in fact 
 require inefficiently small time-steps for nonlinear stability in problems 
 with discontinuities that are of interest to us.

 \begin{figure}[t] 
\includegraphics[scale=.31]{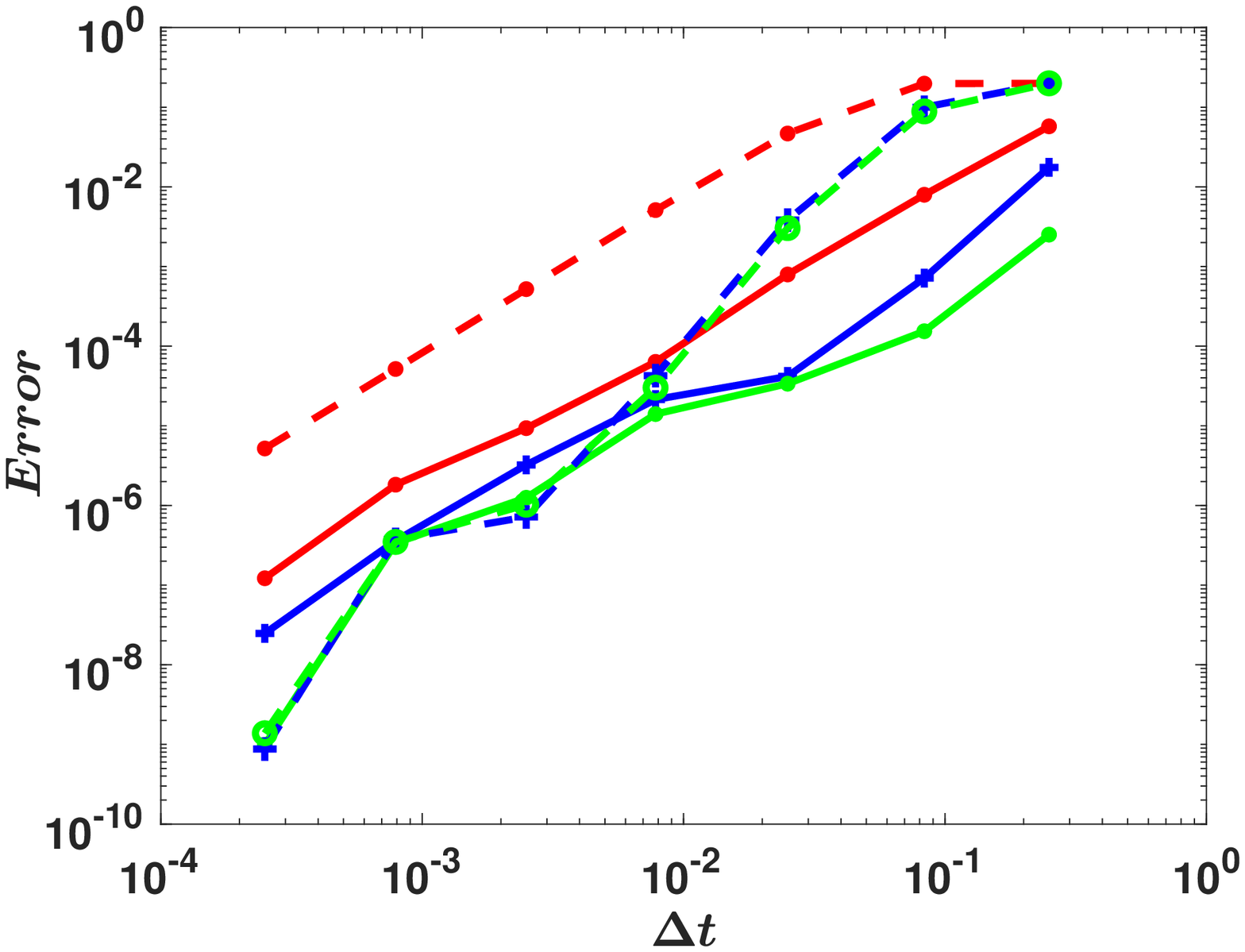}
\includegraphics[scale=.31]{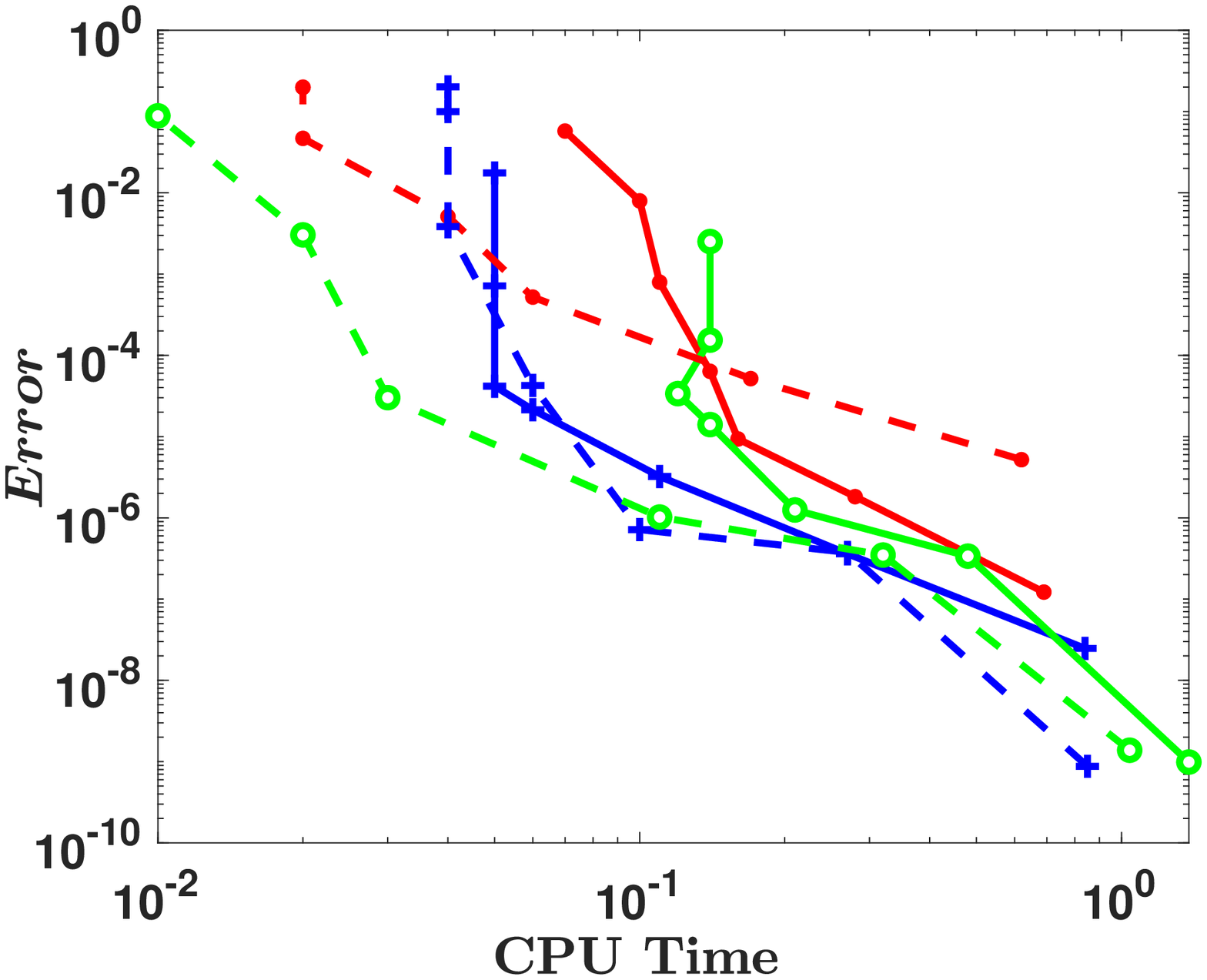}
\caption{ Example 2.
The second order methods are in red, third order in blue, and fourth order in green,
dashed lines represent the ETD methods while solid lines are the eSSPIFRK methods.
}
 \label{fig:Accuracy}
\end{figure}


\subsection{Example 3:  Sharpness of SSP  time-step for a linear problem} 
We consider the linear advection equation with a step function initial condition:
\begin{align}
u_t + a u_x + u_x & = 0 \hspace{.75in}
    u(0,x)  =
\begin{cases}
1, & \text{if } \frac{1}{4} \leq x \leq \frac{3}{4} \\
0, & \text{else } \nonumber
\end{cases}
\end{align}
on the domain $[0,1]$ with periodic boundary conditions.
We use a first-order forward difference for each of the spatial derivatives 
to semi-discretize this problem on a grid  of $0 \leq x \leq 1$ with 1000 points in space and evolve it ten time-steps forward. 

It is known that  this spatial discretization when coupled with forward Euler is TVD, under the time-step restriction
$ \tDtFE  = \frac{1}{a} \Dx$ for the   term $L u \approx a u_x$,
and the restriction $ \DtFE  =  \Dx$ for  the term $ N(u) \approx u_x$.

We measure the total variation of the numerical solution at each stage (to ensure internal stage monotonicity), 
and compare it to the  total variation at the previous stage. We are interested in the size of  time-step $\dt$ at which the total variation
begins to rise. We refer to this value as the {\em observed TVD time-step} $\dt^{TVD}_{obs}$. We
compare this value with the expected TVD time-step dictated by the theory. We call the SSP coefficient
corresponding to the value of the observed TVD time-step the {\em observed SSP coefficient} $\sspcoef_{obs}$.
Note that the expected TVD time-step $ \DtFE  =  \Dx$, so that 
 \[ \sspcoef_{obs} =  \frac{\dt^{TVD}_{obs}}{\DtFE} = \frac{\dt^{TVD}_{obs}}{\Dx} = \lambda^{TVD}_{obs}.\]
 
\begin{figure}[h]
\begin{center}
  \includegraphics[scale=.375]{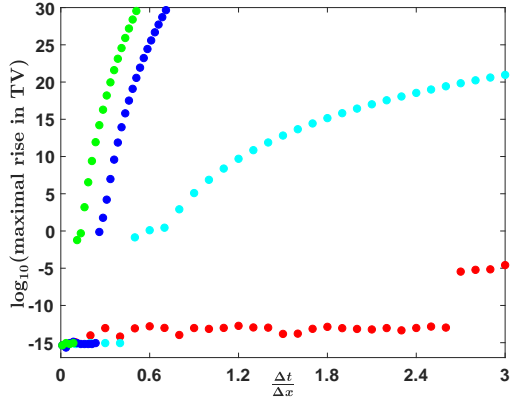} 
 \end{center}
    \caption{ Example 3: Linear advection with a step function initial condition and $a=10$.
The eSSPRK(5,3) method is in blue, the eSSPKG(5,3) in green, the 
 IMEXSSP(5,3,$K=0.1$) method in cyan, and the eSSPIFRK(5,3) in red.
 \label{fig:ex3TVD} }
\end{figure}

\subsubsection{Comparison of integrating methods for wavespeed $a=10$}
First we consider this problem with wavespeed $a=10$.
In Figure \ref{fig:ex3TVD}, we show the observed maximal rise in total variation (on the y-axis)
 when this equation is evolved forward by
a variety of methods using different  values of the Courant number $\lambda = \frac{\dt}{\dx}$ (on the x-axis).

Two methods,  eSSPRK(5,3) (blue) from \cite{SSPbook2011} 
 and  eSSPKG(5,3) (green)     generated in this work,
treat the two spatial terms in the same manner. In other words, they evolve the linear advection problem with 
wavespeed $(1+a)=11$. As expected, the  observed  time-step before the total variation begins to rise is
given by the SSP coefficient scaled by the wavespeed. Our numerical results (represented in Figure \ref{fig:ex3TVD})
show that the eSSPRK(5,3) method has an allowable Courant number 
  $\lambda^{TVD}_{obs} = \frac{\sspcoef}{11}  = \frac{2.6506}{11} = 0.2409$ before the maximal total variation 
  begins to rise. On the other hand,  the SSP  Kinnmark and Gray method eSSPKG(5,3)  has a smaller allowable 
  value of $\lambda^{TVD}_{obs} = \frac{\sspcoef}{11}  = \frac{1}{11} = 0.0909$ 
  before the maximal total variation begins to rise.

 Next, we turn our attention to IMEX methods. As shown in \cite{SSPIMEX}
 both the implicit and explicit terms have to be SSP, and the wavespeed impacts
 the size of the allowable SSP time-step, even though the fast wave is treated implicitly.\footnote{In fact, Table
 \ref{tab:linstab} we observe that an IMEX scheme composed of the Shu-Osher eSSPRK(3,3) for the explicit
 part and an A-stable implicit part does not preserve the TVD property for any time-step in our numerical tests.}
 We use the SSP  implicit-explicit IMEX(5,3,$K=0.1$) method (in cyan) where the slow wave $u_x$ 
 is treated explicitly and the fast wave $10 u_x$ is treated implicitly. 
We use the IMEX(5,3,$K=0.1$) method found in  \cite{SSPIMEX},
which is specially optimized in terms of the allowable SSP time-step for the
 value of $K = \frac{1}{a} = 0.1$.  This method has a $\sspcoef = 0.407$ for the value $a=10$.  
As predicted,  the observed   time-step before the total variation begin to rise is
 $\lambda^{TVD}_{obs} = \frac{\sspcoef}{10} = 0.407$. 
We see here that using an IMEX SSP method where the fast moving wave is treated implicitly does not give us much benefit 
when we are concerned with the TVD behavior of the scheme: it allows us only an increase of 69\% in time-step, at the major cost of 
an implicit solve. 

Finally, we use the eSSPIFRK(5,3) method (in red) resulting from using our eSSPRK+(5,3) method in  formulation
 \eqref{rkIFSO}  where $N(u) = - D u$ and $Lu = - 10 D u$. This method has a SSP coefficient $\sspcoef= 2.635 $
 for any value of $a$ and indeed we observe that the allowable time-step before we see a rise in the total variation is 
 $\dt^{TVD}_{obs} = \frac{\sspcoef} \dx$. This means that the largest allowable time-step for the  eSSPIFRK(5,3) 
 is more than ten times larger than that for the eSSPRK(5,3), and more than six times larger than for the IMEX method. 
 Additionally, this result is produced without much additional cost, as the few 
 matrix exponentials needed are pre-computed  once for the entire simulation. 
 Results of the  allowable SSP values for more methods are presented in Table \ref{tab:linstab}  below.

\subsubsection{Considering different wavespeeds}
The main advantage of the SSP integrating factor Runge--Kutta  schemes is that the allowable SSP time-step is not impacted by the 
wavespeed. In this section, we show how the new eSSPIFRK methods perform equally well for different wavespeeds,
and how other methods (including IMEX methods) do not.

Using $a=0$, we verify the expected TVD time-step for most of the 
methods considered (Table \ref{tab:ex2}). However, we observed that for the methods
eSSPIFRK(3,3)  and eSSPIFRK(5,4),  $\sspcoef_{obs}>\sspcoef$. 
For the eSSPIFRK(3,3) 
method,  $\sspcoef = \frac{3}{4}$ but the observed value is $\sspcoef_{obs} =1$. 
This is easy to understand because for this linear problem with no exponential component ($a=0$),
all methods with the same number of stages as order  ($s=p$) are equivalent. Thus,   for this special case, our
eSSPIFRK(3,3)  method can be re-arranged into the eSSPRK(3,3) Shu-Osher method, and we observe the expected TVD
 time-step for that method. 
 A similar phenomenon occurs  for the eSSPIFRK(5,4) method. In this case,
 we can write the stability polynomial of each stage recursively. If we look at the stability polynomial $P(z)$ of the
 fourth stage, we observe that its third derivative becomes negative at $z=1.5594$, which is precisely the
 observed TVD time-step for this method. When we look into the stages to see where the rise in TV occurs,
 we see that it first happens  in the fourth stage of the first time-step.

  \begin{table} 
 \begin{center}
 \begin{tabular}{|l|c|cccc|} \hline
 Method  &   $\sspcoef $  &    \multicolumn{4}{|c|}{$\lambda^{TVD}_{obs}$} \\ 
 & & for $a=0$ &   $a=1.0$ &  $a=10$& $a=20$ \\ \hline
 eSSPIFRK(2,2) &  1 &  1 & 1 & 1 & 1 \\
eSSPIFRK(9,2) &   8 &  8 & 8 & 8 & 8 \\
eSSPIFRK(3,3) & 3/4 &  1 & 3/2 & 3/2 & 3/2 \\
eSSPIFRK(4,3) & 20/11 &  20/11 & 20/11 & 20/11 & 20/11 \\
eSSPIFRK(9,3) & 6 &  6 & 6 & 6 & 6 \\
eSSPIFRK(5,4) & 1.346 & 1.5594 &2.158 &2.158 &2.158 \\
eSSPIFRK(6,4)&  2.273 &  2.273 &  2.273 &  2.273 &  2.273 \\
eSSPIFRK(9,4) & 4.306 & 4.306 & 4.306 & 4.306 & 4.306  \\ \hline
 \end{tabular}
 \end{center}
 \caption{ \label{tab:ex2} The observed SSP coefficient compared to the predicted SSP coefficient
 for Example 3 with various wavespeeds $a$. The value of $a$ does not negatively impact the observed SSP coefficient.  }
 \end{table}

 Next, we consider various values of $a > 0$. The results in Table \ref{tab:ex2} confirm that 
{\bf the value of $a$ does not negatively impact the observed SSP coefficient for this linear example.}\footnote{We note that for 
values larger than $a=20$, the   significant damping that occurs due to the exponential masks the oscillation 
and its associated rise in total variation. }
 In these cases too, we observe that for most methods the SSP condition is sharp, i.e, $\sspcoef_{obs} = \sspcoef$ 
 for all values of $a$. The observed TVD time-step for   eSSPIFRK(3,3)  and
 eSSPIFRK(5,4)  are, once again, larger than expected. 
For the  eSSPIFRK(3,3)  method $\sspcoef_{obs}=\frac{3}{2}$ is a result of the
 rise in TV from the first stage, which has a step-size $\frac{2}{3} \dt$.

Finally, we compare the eSSPIFRK(4,3) method to the explicit SSP Runge--Kutta method and to IMEX methods.
Table \ref{tab:ex2b} shows that the  value of $a$ does not negatively impact the observed SSP coefficient
 for the eSSPIFRK method. When the eSSPRK(4,3) method is applied to the linear advection problem with 
 wavespeed $a+1$, the observed  SSP coefficient matches the predicted
 \[ \lambda^{TVD}_{obs} = \frac{\sspcoef_{obs} }{a+1} =   \frac{\sspcoef }{a+1}   = \frac{2}{a+1}\]
 as shown in Table \ref{tab:ex2b}. We also show the observed SSP coefficient for the SSP IMEX methods
IMEXSSP(4,3,K) from  \cite{SSPIMEX}. 
These methods have SSP explicit and implicit parts and were optimized for the SSP step size 
for each value of $K=\frac{1}{a}$
in \cite{SSPIMEX}.  As we expect from SSP theory, the observed value of $\lambda$ before a rise in total variation
occurs decays linearly as the wavespeed $a$ rises.

\begin{table} 
 \begin{center}
 \begin{tabular}{|l|cccccc|} \hline
 Method    &  \multicolumn{6}{|c|}{$\lambda^{TVD}_{obs}$} \\ 
 &  $a=0$ &   $a=1.0$ & $a=2$  &$a=10$& $a=20$ & $a=100$ \\ \hline
eSSPIFRK(4,3) &   1.818 & 1.818& 1.818& 1.818& 1.818& 4.200  \\
SSPIMEX(4,3,K ) &  2.000 & 1.476 & 1.192 & 0.310 & 0.162 & 0.033 \\
eSSPRK(4,3) & 2.000 & 1.000& 0.666 & 0.181 & 0.0952 & 0.019 \\ \hline
 \end{tabular}
 \end{center}
 \caption{ \label{tab:ex2b} The observed SSP coefficients for an eSSPIFRK method, IMEX method, and 
 explicit Runge--Kutta method
 for Example 3 with various wavespeeds $a$. The value of $a$ does not negatively impact the observed SSP coefficient
 for the eSSPIFRK method, but it does for the IMEX and the explicit Runge--Kutta methods.  }
 \end{table}

\subsubsection{Comparison to linear stability properties}
In Figure \ref{fig:ex3TVD} we notice that for each method, as the Courant number $\lambda = \frac{\dt}{\dx}$ gets large enough,
the maximal rise in total variation jumps up. For some methods, this happens for very small values of $\lambda$,
while for others, it happens when $\lambda$ is larger. It is  interesting to see how this $\lambda^{TVD}$ allowable for the
TVD property compares to the CFL number $\lambda^{L_2}$ allowed for linear stability for this problem.
Once again, we use $a=10$ and compare the allowable
time-step for  linear stability and the allowable predicted and observed  TVD  time-step  
in Table \ref{tab:linstab}.  We  evaluate the allowable time-step for linear stability of a given method  for this
particular problem by calculating the 
stability polynomial for these operators and determining at which value of $\lambda^{L_2}$  
the $L_2$ norm of the resulting matrix becomes greater than one.

We notice that  linear stability does not provide any prediction on the TVD behavior of these methods.
The most striking example of this is when looking at the IMEX methods; in terms of linear stability,
these can clearly eliminate the constraint coming
from the fast wave (e.g., IMEXSSP(3,3,$K=\infty$)), and provide a large region of linear stability. However, this method fails
to be SSP for any positive time-step. Even a specially designed and optimized SSP IMEX method did not even 
double the allowable time-step for TVD. 

These results underscore the fact that when solving problems 
with discontinuities, the relevant time step restriction is dictated by $\lambda^{TVD}$ rather than $\lambda^{L_2}$.
The most notable result from this table is that all the integrating factor methods have a very large $L_2$ linear stability region for this problem.
In fact, we tested them up to $\lambda=27$, which is ten times larger than the largest allowable time-step for TVD, 
and they were still stable at this value. Clearly, the integrating approach has advantages for linear stability as well as 
for strong stability preservation.

We also note from the results in Table \ref{tab:linstab} that although the Kinnmark and Gray methods are designed to have larger 
imaginary axis linear stability regions, for our problem this does not give them an advantage even for linear stability.
This is easily understood when we consider that the eigenvalues of our differentiation operator are a circle in the 
complex plane, and so the linear stability regions of the regular eSSPRK methods are better suited than those
of the eSSPKG methods of Kinnmark and Gray (as seen in Figure \ref{KGstability}).

\smallskip

\begin{table}
\begin{center}
\begin{tabular}{|l|l|ll|} \hline
Method & $\lambda^{L_2}_{obs}$ &  $ \lambda^{TVD}_{pred}$  &   $\lambda^{TVD}_{obs}$ \\ 
& & &  \vspace{0.01pt}  \\ \hline
eSSPRK(3,3)	                &	0.114	& 1/11	& 0.090 \\
eSSPKG(3,3)			&	0.114	& 1/11	& 0.090 \\
eSSPRK+(3,3)                  & 	0.114	& 3/44	& 0.090 \\
eSSPKG+(3,3)	                &	0.114	& 1/11	& 0.090 \\
IMEXSSP(3,3,$K=0.1$))	&     	0.448	& 0.149	& 0.236 \\
IMEXSSP(3,3,$K=\infty$))& 	1.198	& 0.000	& 0.000 \\  
eSSPIFKG(3,3) 		&  * 			& 0.750	& 1.500 \\
eSSPIFRK(3,3)			&  * 			&  0.750	& 1.500 \\ \hline
eSSPRK(5,3)	 		&	0.260	&0.240 	& 0.240 \\
eSSPKG(5,3)  			& 	0.138	& 1/11 	&  0.090\\
eSSPRK+(5,3)			&	0.261	& 0.239	& 0.239 \\ 
eSSPKG+(5,3) 			&	0.138	& 1/11	& 0.090 \\
IMEXSSP(5,3,$K=0.1$)	& 	0.683	& 0.407	& 0.407 \\  
eSSPIFKG(5,3) 		&  * 			& 0.875	& 1.487 \\
eSSPIFRK(5,3)			& * 			& 2.635 	& 2.635 \\ \hline
eSSPRK(6,4)			&	0.273	& 0.208	&0.208 \\
eSSPKG(6,4)			&	0.146	& 1/11  	&  0.090 \\
eSSPRK+(6,4)			&	0.270	& 0.206	& 0.206 \\
eSSPKG+(6,4)			& 	0.146	& 0.071	& 0.090 \\ 
eSSPIFKG(6,4)			& * 			& 0.785	& 1.805 \\
eSSPIFRK(6,4) 		& * 			& 2.273	& 2.273 \\  \hline
\end{tabular}
\end{center}
\caption{The values of the observed CFL number $\lambda^{L_2} = \frac{\dt}{\dx}$ required for $L_2$ 
linear stability for Example 3: $u_t + u_x + 10 u_x = 0$, compared to the predicted  and observed  values $\lambda^{TVD}$.
An * indicates that these methods were linearly $L_2$ stable for the largest values tested, $\lambda \leq 27$.
} \label{tab:linstab}
\end{table}

\subsection{Example 4:  Sharpness of SSP  time-step for a nonlinear problem}
  Consider the equation:
  \begin{align}
u_t + a u_x +  \left( \frac{1}{2} u^2 \right)_x & = 0 \hspace{.75in}
    u(0,x)  =
\begin{cases}
1, & \text{if } 0 \leq x \leq 1/2 \\
0, & \text{if } x>1/2 \nonumber
\end{cases}
\end{align}
on the domain $[0,1]$ with periodic boundary conditions.
We used a first-order upwind difference to semi-discretize this linear term, and a fifth order WENO finite difference for the 
nonlinear terms. We solved this problem on a spatial grid  with 400 points and evolved it forward 25 time steps
using $\dt = \lambda \dx$.
We measured the total variation at each stage, and calculated the maximal rise in total
variation over each stage for these 25 time steps. 

In  Figure \ref{fig:WENOBurgers_adv} (left) we use the value of $a=5$ and 
graph the $log_{10}$ of the maximal rise in total variation
versus the ratio $\lambda = \frac{\dt}{\Delta x}$. We observe that our eSSPIFRK(3,3) method
\eqref{SSPIFRK33} maintains a very small maximal rise in total variation until close to $\lambda = 0.8$,
while the fully explicit third order Shu-Osher method begins to feature a large rise in total variation
for a much smaller value of $\lambda=.15 \approx \frac{1}{1+a}$.
In contrast, the three-stage third order ETD Runge--Kutta   \cite{CoxMatthews2002} 
and the  integrating factor method based on the Shu-Osher method \eqref{SOIF}
both have a maximal rise in total variation that increases rapidly with $\lambda$.

\newpage

In  Figure \ref{fig:WENOBurgers_adv}  (right) we show a similar study using wavespeed  $a=10$ and fourth order methods.
versus the ratio $\lambda = \frac{\dt}{\Delta x}$. We observe that our eSSPIFRK(5,4), 
eSSPIFRK(6,4), and eSSPIFRK(9,4)  methods maintain a very small maximal rise in total variation until close to 
$\lambda = 1.06, 1.21, 2.41$, respectively.
In comparison, the  SSP IMEX (5,4) method features an observed $\lambda$ value of $\lambda_{obs} = .25$, the fully explicit SSPRK(10,4)
has $\lambda_{obs} = 0.58$, and the maximal total variation from the simulation using the ETDRK4 
method starts rising rapidly from the smallest value of $\lambda$.

\begin{figure}
\begin{center}
  \includegraphics[scale=.3]{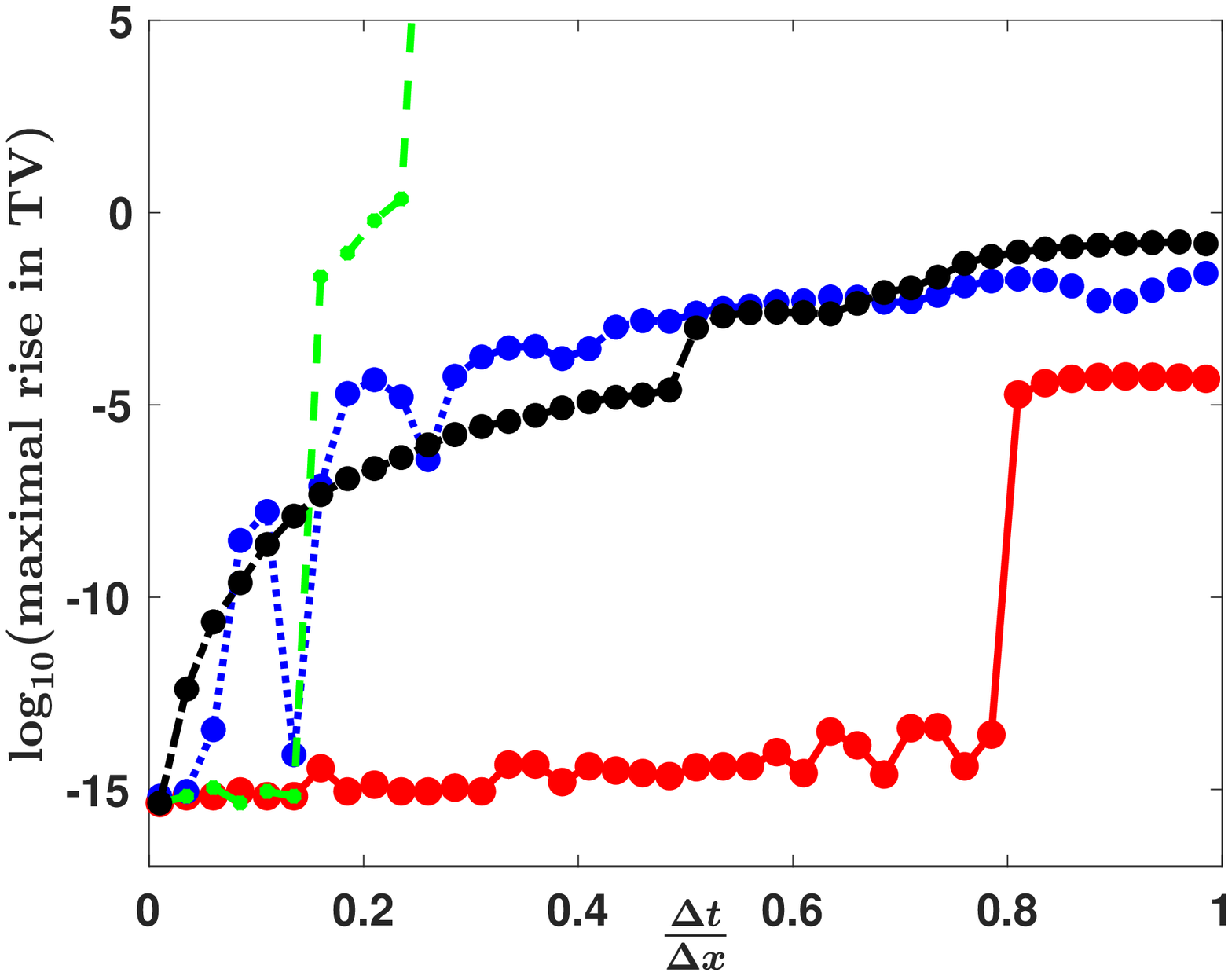} 
  \includegraphics[scale=.3]{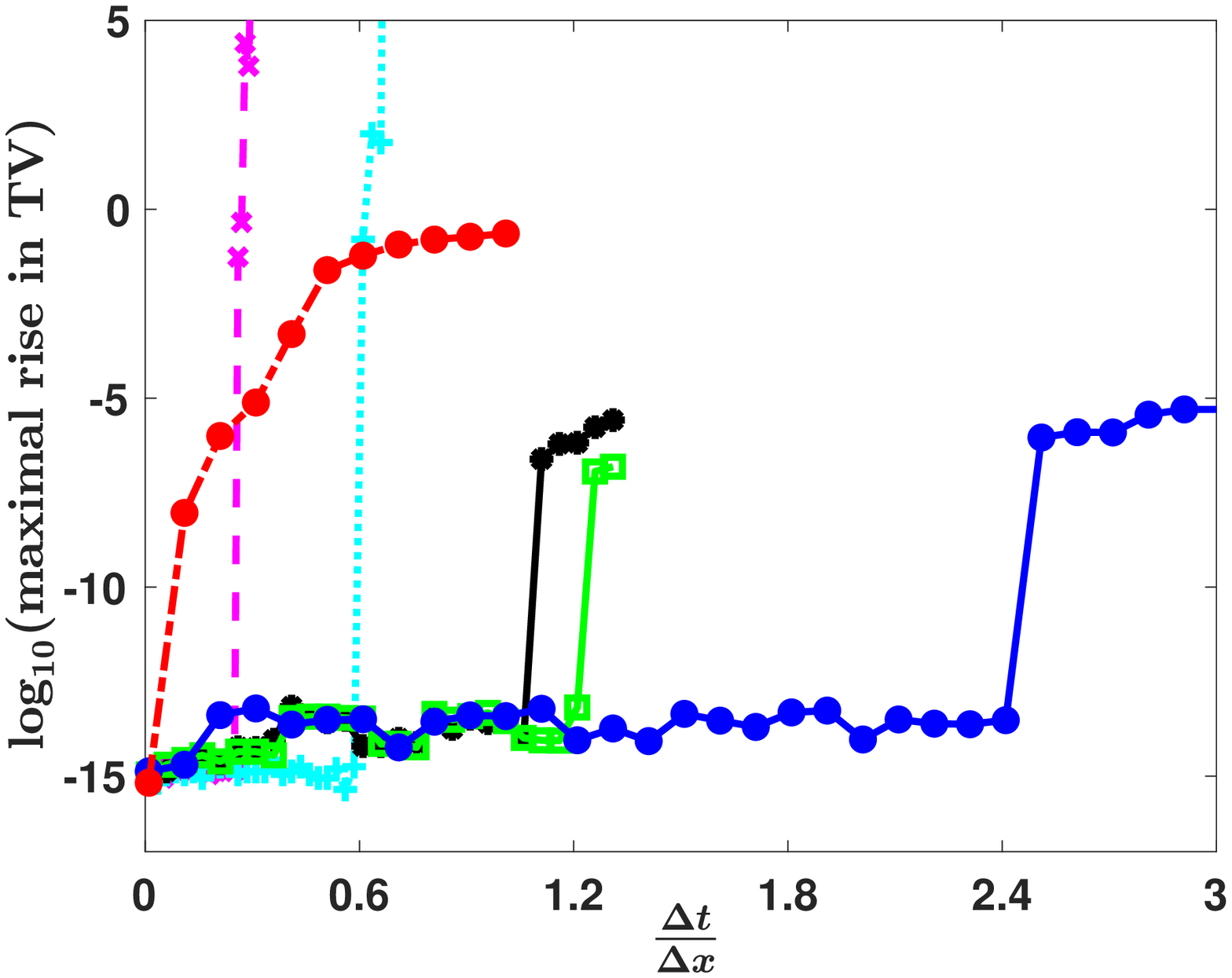} 
    \caption{ Example 4.
Left:  $a=5$. The red solid line is the eSSPIFRK(3,3) method,
the blue dotted line is the IFRK method based on the eSSPRK(3,3) Shu-Osher method,
the black dash-dot line is the ETDRK3 method, and the green dashed line is the 
third order explicit  eSSPRK(3,3)  Shu-Osher method. Right: $a=10$.
The dashed magenta line with x markers is the SSP  IMEX (5,4) method.
The dash-dot red line with filled circle markers is the  ETDRK4 method.
The dotted cyan line with + markers is the SSPRK(10,4) method by
Ketcheson
The solid lines are the SSPIFRK methods: the 
black line with star markers for (5,4), green line with square markers for (6,4) and  blue line with circle markers for (9,4).
 \label{fig:WENOBurgers_adv} }
 \end{center}
\end{figure}

\section{Conclusions}  \label{sec:conclusions}
This is the first work to consider strong stability  preserving integrating factor Runge--Kutta methods.
In this work we presented sufficient conditions for preservation of strong stability for integrating factor Runge--Kutta methods.
These eSSPIFRK methods are based on eSSPRK methods with non-decreasing abscissas. We used these conditions
to develop an optimization problem which we used to find such eSSPRK$^+$ methods. 
We then showed that these eSSPIFRK methods perform in practice as
expected, significantly out-performing the implicit-explicit (IMEX) SSP Runge--Kutta methods and the ETD methods of Cox and Matthews
on problems that require the SSP property.

\bigskip

{\bf Acknowledgment.} 
This publication is based on work supported by  AFOSR grant FA9550-15-1-0235.
The authors thank David Ketcheson and the anonymous referees for their helpful remarks on this work.


\begin{thebibliography}{10}

\bibitem{AlMohyHigham}
{\sc A.H. Al Mohy and  N.J. Higham},
{\em Computing the action of the matrix exponential of a matrix,
with an application to exponential integrators},
SIAM Journal on Scientific Computing 33(2) (2011), pp.~488--511.


\bibitem{baiotti2005}
{\sc L. Baiotti, I. Hawke, P.J. Montero, F. Loffler, L. Rezzolla,
  N. Stergioulas, J.A. Font and E. Seidel}.
{\em Three-dimensional relativistic simulations of rotating neutron-star
  collapse to a {K}err black hole.}
Physical Review D 71 (2005), pp.~24--35.


\bibitem{balbas2005}
{\sc  J. Balb{\'a}s and E. Tadmor.}
{\em  A central differencing simulation of the {O}rszag-{T}ang vortex system},
IEEE Transactions on Plasma Science 33 (2005), pp.470--471. 


\bibitem{bassano2003}
{\sc E. Bassano},
{\em Numerical simulation of thermo-solutal-capillary migration of a
  dissolving drop in a cavity},
International Journal for Numerical Methods in Fluids 41 (2003), pp.~765--788.


\bibitem{EXPINTpackage}
{\sc H.~Berland, B.~Skaflestad, and W.~M. Wright}, {\em Expint -- a
  {M}{A}{T}{L}{A}{B} package for exponential integrators}, 2005.

\bibitem{EXPINT}
{\sc H.~Berland, B.~Skaflestad, and W.~M. Wright}, {\em Expint -- a
  {M}{A}{T}{L}{A}{B} package for exponential integrators}, ACM Transactions in
  Mathematical Software, 33 (2007).

\bibitem{Butcher}
{\sc J.C. Butcher},
{\em The numerical analysis of ordinary differential equations: Runge-Kutta and general linear methods},
Wiley-Interscience New York, NY, USA 1987 


\bibitem{caiden2001}
{\sc R. Caiden, R.P. Fedkiw and C. Anderson},
{\em A numerical method for two-phase flow consisting of separate
  compressible and incompressible regions},
Journal of Computational Physics 166 (2001), pp.~1--27.

\bibitem{carrillo2003}
{\sc J. Carrillo, I.M. Gamba, A. Majorana and C.-W. Shu},
{\em A {WENO}-solver for the transients of {B}oltzmann-{P}oisson system for
  semiconductor devices: performance and comparisons with {M}onte {C}arlo methods},
Journal of Computational Physics 184 (2003), pp.~498--525.

\bibitem{cheng2003}
{\sc  L.-T. Cheng, H. Liu and S. Osher},
{\em Computational high-frequency wave propagation using the level set
  method, with applications to the semi-classical limit of {S}chrodinger  equations},
Communications in Mathematical Sciences 1 (2003), pp.~593--621.

\bibitem{cheruvu2007}
{\sc V. Cheruvu, R.D. Nair and H.M. Tufo},
{\em  A spectral finite volume transport scheme on the cubed-sphere},
Applied Numerical Mathematics 57 (2007), pp.~1021--1032.

\bibitem{cockburn1989}
{\sc B. Cockburn and C.-W. Shu},
{\em  {TVB} {R}unge-{K}utta local projection discontinuous {G}alerkin finite
  element method for conservation laws {II}: general framework},
Mathematics of Computation 52 (1989), pp.~411--435.

\bibitem{cockburn2004}
{\sc B. Cockburn, F. Li and C.-W. Shu},
{\em  Locally divergence-free discontinuous {G}alerkin methods for the
  {M}axwell equations},
Journal of Computational Physics 194 (2004), pp.~588--610.


\bibitem{cockburn2005}
{\sc B. Cockburn, J. Qian, F. Reitich and J. Wang},
{\em  An accurate spectral/discontinuous finite-element formulation of a
  phase-space-based level set approach to geometrical optics},
Journal of Computational Physics 208 (2005), pp.~175--195.


\bibitem{SSPIMEX} 
{\sc S. Conde, S. Gottlieb, Z. Grant, J.N. Shadid},
{\em Implicit and Implicit-Explicit Strong Stability Preserving RungeÐKutta Methods with High Linear Order},
 Journal of Scientific Computing 73(2)  (2017), pp.~667--690.


\bibitem{CoxMatthews2002}
{\sc S.~Cox and P.~Matthews}, {\em Exponential time differencing for stiff
  systems}, Journal of Computational Physics, 176 (2002), pp.~430--455.

\bibitem{delzanna2002}
{\sc  L. Del Zanna and N. Bucciantini},
{\em An efficient shock-capturing central-type scheme for multidimensional
  relativistic flows: I. hydrodynamics},
Astronomy and Astrophysics  390 (2002), pp.~1177--1186.


\bibitem{enright2002}
D. Enright, R. Fedkiw, J. Ferziger and I. Mitchell.
\newblock A hybrid particle level set method for improved interface capturing.
\newblock {\em Journal of Computational Physics}, 183:83--116, 2002.


\bibitem{feng2004}
{\sc L.-L. Feng, C.-W. Shu and M. Zhang},
{\em  A hybrid cosmological hydrodynamic/n-body code based on a weighted
  essentially nonoscillatory scheme},
The Astrophysical Journal  612 (2004), pp.~1--13.

\bibitem{ferracina2004}
{\sc L. Ferracina and M.N. Spijker},
{\em Stepsize restrictions for the total-variation-diminishing property in general {R}unge-{K}utta methods},
 SIAM Journal of Numerical Analysis 42 (2004), pp.~1073-1093.

\bibitem{ferracina2005}
{\sc L. Ferracina and M.N. Spijker},
{\em An extension and analysis of the {S}hu-{O}sher representation of
  {R}unge-{K}utta methods},
Mathematics of Computation 249(2005), pp.~201--219.

\bibitem{ferracina2008}
{\sc  L. Ferracina and M.N. Spijker},
{\em  Strong stability of singly-diagonally-implicit {R}unge-{K}utta methods},
Applied Numerical Mathematics 2008


\bibitem{GaudreaultRainwaterTokman}
{\sc   S. Gaudreault, G. Rainwater, M. Tokman},
{\em KIOPS: A fast adaptive Krylov subspace solver for exponential integrators},
	arXiv:1804.05126 [math.NA] (2018).

\bibitem{HesthavenGottlieb2001}
{\sc  D. Gottlieb and J.S. Hesthaven},
{\em Spectral methods for hyperbolic problems},
Journal of Computational and Applied Mathematics, 
128 (2001),  pp.~83-131. 

\bibitem{GottliebTadmor1991}
{\sc D. Gottlieb and E. Tadmor}
{\em The CFL condition for spectral approximations to hyperbolic initial-boundary value problems},
Mathematics of Computation 56 (1991), pp.~565--588. 


\bibitem{SSPIFgithub}
{\sc S.~Gottlieb, Z.~Grant, and L.~Isherwood}, {\em Optimized strong stability
  preserving integrating factor {R}unge--{K}utta methods}.
\newblock \url{https://github.com/SSPmethods/SSPIFRK-methods}.

\bibitem{SSPbook2011}
{\sc S.~Gottlieb, D.~I. Ketcheson, and C.-W. Shu}, {\em Strong Stability
  Preserving Runge--Kutta and Multistep Time Discretizations}, World Scientific
  Press, 2011.

\bibitem{gottliebshu1998}
{\sc S.~Gottlieb and C.-W. Shu}, {\em Total variation diminishing Runge--Kutta
  methods}, Mathematics of Computation, 67 (1998), pp.~73--85.

\bibitem{gottlieb2001}
{\sc S.~Gottlieb, C.-W. Shu, and E.~Tadmor}, 
{\em Strong Stability Preserving  High-Order Time Discretization Methods}, 
  SIAM Review, 43 (2001), pp.~89--112.

\bibitem{harten1983}
{\sc A. Harten},
{\em High resolution schemes for hyperbolic conservation laws},
Journal of Computational Physics, 49 (1983), pp.~357--393.


\bibitem{HesthavenCLbook}
{\sc J.S. Hesthaven},
{\em Numerical methods for conservation laws: From analysis to algorithms},
 SIAM Publishing, Philadelphia (2017).

\bibitem{HGGbook}
{\sc J.S. Hesthaven, S. Gottlieb, and D. Gottlieb},
{\em Spectral Methods for Time Dependent Problems},
Cambridge Monographs on Applied and Computational Mathematics (No. 21)
Cambridge University Press (2006).  


\bibitem{higueras2004a}
{\sc I. Higueras},
{\em On strong stability preserving time discretization methods},
Journal of Scientific Computing  21 (2004), pp.~193--223.

\bibitem{higueras2005a}
{\sc I. Higueras},
{\em Representations of {R}unge-{K}utta methods and strong stability
  preserving methods},
SIAM Journal on Numerical Analysis  43 (2005), pp.~924--948.


\bibitem{HuAdamsShu2012}
{\sc X. Y. Hu, N. A. Adams, and C.-W. Shu},
{\em Positivity-preserving method for high-order conservative schemes solving compressible Euler equations},
Journal of Computational Physics 242 (2013), pp.~169--180.


\bibitem{hundsdorfer2003}
{\sc W. Hundsdorfer, S.J. Ruuth and R.J. Spiteri},
{\em Monotonicity-preserving linear multistep methods},
SIAM Journal on Numerical Analysis 41 (2003), pp.~605--623.

\bibitem{jin2005}
{\sc S. Jin, H. Liu, S. Osher and Y.-H.R. Tsai},
{\em Computing multivalued physical observables for the semiclassical
  limit of the {S}chrodinger equation},
Journal of Computational Physics 205 (2005), pp.~222--241.


\bibitem{KennedyCarpenterLewis2000}
{\sc C.A. Kennedy, M.H. Carpenter, and R.M. Lewis},
{\em Low-storage explicit Runge--Kutta schemes for the compressible Navier-Stokes
equations},
Applied Numerical Mathematics 35 (2000), pp.~177-219.

\bibitem{ketcheson2008}
{\sc D.~I. Ketcheson}, {\em Highly efficient strong stability preserving
  {R}unge--{K}utta methods with low-storage implementations}, SIAM Journal on
  Scientific Computing, 30 (2008), pp.~2113--2136.

\bibitem{ketcheson2009}
{\sc D.I. Ketcheson, C.B. Macdonald and S. Gottlieb},
{\em Optimal implicit strong stability preserving {R}unge-{K}utta methods},
Applied Numerical Mathematics 52 (2009), pp.~373--392.


\bibitem{ketcheson2009a}
 {\sc D. I. Ketcheson},
 {\em Computation of optimal monotonicity preserving general linear methods},
 Mathematics of Computation 78 (2009), pp.~1497--1513.

\bibitem{ketcheson2011}
 {\sc D. I. Ketcheson},
{\em Step sizes for strong stability preservation with downwind-biased operators},
SIAM Journal on Numerical Analysis 49 (4) (2011), pp.~1649--1660.

\bibitem{KG1984}
{\sc Kinnmark and Gray},
{\em One step integration methods with maximum stability regions},
Mathematics and Computers in Simulation 26 (2) (1984), 
pp.~87--92.


\bibitem{kraaijevanger1991}
{\sc J.~F. B.~M. Kraaijevanger}, 
{\em Contractivity of {R}unge--{K}utta methods}, 
BIT, 31 (1991), pp.~482--528.


\bibitem{KubatkoKetcheson}
{\sc E.J.  Kubatko, B. A. Yeager, and D. I. Ketcheson}, 
{\em Optimal strong-stability-preserving RungeÐKutta time discretizations for discontinuous Galerkin methods}, 
Journal of Scientific Computing 60(2)  (2014), pp.~313--344.


\bibitem{kurganov2000}
{\sc A. Kurganov and E. Tadmor},
{\em New high-resolution schemes for nonlinear conservation laws and
  convection-diffusion equations},
Journal of Computational Physics 160 (2000), pp.~241--282.


\bibitem{labrunie2004}
{\sc S. Labrunie, J. Carrillo and P. Bertrand},
{\em Numerical study on hydrodynamic and quasi-neutral approximations for
  collisionless two-species plasmas},
Journal of Computational Physics 200 (2004), pp.~267--298.


\bibitem{lawson1967}
{\sc J. D. Lawson},
{\em Generalized Runge-Kutta Processes for Stable Systems with Large Lipschitz Constants},
SIAM Journal on  Numerical Analysis, 4(3) (1967), pp.~372--380. 

\bibitem{Lax2006}
{\sc P.D. Lax}, {\em  Gibbs Phenomena},
Journal of Scientific Computing, 28 (2006), pp.~445--449.


\bibitem{lenferink1991}
{\sc H.W.J. Lenferink},
{\em Contractivity-preserving implicit linear multistep methods},
Mathematics of Computation 56 (1991), pp.~177--199.


\bibitem{LeVequeBook}
{\sc R. J. LeVeque},
{\em Numerical Methods for Conservation Laws},
ETH Lectures in Mathematics Series, Birkhauser-Verlag, (1990).


\bibitem{LevyTadmor1998}
{\sc D. Levy and E. Tadmor},
{\em From Semi-Discrete to Fully-Discrete: The Stability of Runge-Kutta Schemes by the Energy Method}, 
SIAM Review, 40(1) (1998), pp.~40--73.


\bibitem{liu1994}
{\sc X.-D. Liu, S. Osher and T. Chan},
{\em Weighted essentially non-oscillatory schemes},
Journal of Computational Physics 115(1) (1994), pp.~200--212.

\bibitem{Ma1989}
{\sc H. Ma},
{\em Chebyshev--Legendre super spectral viscosity method for nonlinear conservation laws},
SIAM Journal on Numerical Analysis 35(3) (1998), pp.~893--908.


\bibitem{MadayTadmor1989}
{\sc Y. Maday and E. Tadmor},
{\em Analysis of the spectral vanishing viscosity method for periodic conservation laws},
SIAM Journal on Numerical Analysis 26 (1989), pp.~854-870.


\bibitem{majda1978}
{\sc A. Majda and S. Osher},
{\em  A systematic approach for correcting nonlinear instabilities. the
  {L}ax-{W}endroff scheme for scalar conservation laws},
Numerische Mathematik 30 (1978), pp.~429--452.

\bibitem{mignone2005}
{\sc A. Mignone},
{\em The dynamics of radiative shock waves: linear and nonlinear  evolution},
The Astrophysical Journal 626 (2005), pp.~373--388.

\bibitem{dubious}
{\sc G. Golub and C. Van Loan},
{\em Nineteen Dubious Ways to Compute the Exponential of a Matrix, Twenty-Five Years Later}.
SIAM Review  45 (2003), pp.~801--836.

\bibitem{NiesenWright}
{\sc J. Niesen and W.M. Wright},
{\em Algorithm 919: A Krylov subspace algorithm for evaluating the $\phi$-functions appearing
in exponential integrators},
ACM Transactions on Mathematical Software 38(3) (2012), pp.~22. 


\bibitem{osher1984}
{\sc S. Osher and S. Chakravarthy},
{\em High resolution schemes and the entropy condition},
 SIAM Journal on Numerical Analysis  21 (1984), pp.~955--984.
 

\bibitem{pantano2007}
{\sc C. Pantano, R. Deiterding, D.J. Hill and D.I. Pullin},
{\em A low numerical dissipation patch-based adaptive mesh refinement
  method for large-eddy simulation of compressible flows},
Journal of Computational Physics 221 (2007), pp.~63--87.


\bibitem{patel2005}
{\sc  S. Patel and D. Drikakis},
{\em Effects of preconditioning on the accuracy and efficiency of
  incompressible flows},
International Journal for Numerical Methods in Fluids 47 (2005), pp.~963--970.


\bibitem{peng1999}
{\sc  D. Peng, B. Merriman, S. Osher, H. Zhao and M. Kang},
{\em A {PDE}-based fast local level set method},
Journal of Computational Physics 155 (1999), pp.~410--438.


\bibitem{ReddyTrefethen1990}
{\sc S.~C.Reddy and L.~N.Trefethen},
{\em   Lax stability of fully discrete  spectral methods via stability regions and pseudo-eigenvalues},
Computer Methods in Applied Mechanics and Engineering
80(1Ð3) (1990), pp.~147--164.


\bibitem{ruuth2001}
{\sc S.~J. Ruuth and R.~J. Spiteri}, {\em Two barriers on
  strong-stability-preserving time discretization methods}, Journal of
  Scientific Computation, 17 (2002), pp.~211--220.

\bibitem{shu1987}
{\sc C.-W. Shu}, 
{\em TVB uniformly high-order schemes for conservation laws}, 
Mathematics of Computation 49 (1987), pp.~105--121.


\bibitem{shu1988b}
{\sc C.-W. Shu}, {\em Total-variation diminishing time discretizations}, SIAM
  Journal on Scientific Statistical Computing 9 (1988), pp.~1073--1084.

\bibitem{shu1988}
{\sc C.-W. Shu and S.~Osher}, {\em Efficient implementation of essentially
  non-oscillatory shock-capturing schemes}, Journal of Computational Physics
  77 (1988), pp.~439--471.


\bibitem{Sidje}
{\sc R.B. Sidje},
{\em EXPOKIT: A software package for computing matrix exponentials},
ACM Transactions on Mathematical Software 24(1) (1998), pp.~130--156.


\bibitem{spijker1983}
{\sc M.N. Spijker},
{\em Contractivity in the numerical solution of initial value problems},
Numerische Mathematik 42 (1983), pp.~271--290.

\bibitem{spijker2007}
{\sc M.~Spijker}, {\em Stepsize conditions for general monotonicity in
  numerical initial value problems}, SIAM Journal on Numerical Analysis, 
  45 (2008), pp.~1226--1245.

\bibitem{SpiteriRuuth2002}
{\sc R.~J. Spiteri and S.~J. Ruuth}, {\em A new class of optimal high-order
  strong-stability-preserving time discretization methods}, SIAM Journal on
  Numerical Analysis, 40 (2002), pp.~469--491.

\bibitem{strang1964}
{\sc G. Strang},
{\em Accurate partial difference methods II: nonlinear problems},
Numerische Mathematik  6 (1964), pp.~37--46.

\bibitem{strikwerda1989}
{\sc J.C. Strikwerda},
{\em Finite Difference Schemes and Partial Differential Equations},
Cole Mathematics Series. Wadsworth and Brooks, California, 1989.

\bibitem{sun2006}
{\sc Y. Sun, Z.J. Wang and Y. Liu},
{\em Spectral (finite) volume method for conservation laws on unstructured
  grids {VI}: {E}xtension to viscous flow.},
Journal of Computational Physics, 215 (2006), pp.~41--58.

\bibitem{sweby1984}
{\sc  P.K. Sweby},
{\em High resolution schemes using flux limiters for hyperbolic
  conservation laws},
SIAM Journal on Numerical Analysis  21 (1984), pp.~995--1011.

\bibitem{tadmor1998}
{\sc E. Tadmor},
{\em Advanced Numerical Approximation of Nonlinear Hyperbolic
  Equations,'' Lectures Notes from CIME Course Cetraro, Italy, 1997},
  pp~1--150  in  Approximate solutions of nonlinear conservation laws, 
Number 1697 in Lecture Notes in Mathematics. Springer-Verlag, 1998,


\bibitem{Tadmor1990}
{\sc E. Tadmor},
{\em Shock capturing by the spectral viscosity method},
Computer Methods in Applied Mechanics and Engineering 80 (1990), pp.~197-208.


\bibitem{tanguay2003}
{\sc  M. Tanguay and T. Colonius},
{\em  Progress in modeling and simulation of shock wave lithotripsy},
 In  Fifth International Symposium on cavitation (CAV2003), OS-2-1-010, 2003.
  
 \bibitem{vanderHouwen1972}
{\sc  P.J. van der Houwen},
{\em  Explicit Runge--Kutta formulas with increased stability boundaries},
Numerische Mathematik 20  (1972), pp.~149-164.
  
  \bibitem{wang2005}
{\sc  Z.J. Wang and Y. Liu},
{\em The spectral difference method for the 2{D} {E}uler equations on
  unstructured grids},
 In { 17th AIAA Computational Fluid Dynamics Conference}, AIAA,  2005.

\bibitem{wang2007a}
{\sc Z.J. Wang, Y. Liu, G. May and A. Jameson},
{\em Spectral difference method for unstructured grids {II}: {E}xtension to
  the {E}uler equations}, Journal of Scientific Computing 32(1) (2007), pp.~45--71.

\bibitem{Williamson1980}
{\sc J.~H. Williamson},
{\em Low storage Runge--Kutta schemes},
Journal of Computational Physics, 35 (1980), pp.~48-56.

  \bibitem{zhang2006}
{\sc W. Zhang and A.I. MacFayden},
{\em {RAM}: A relativistic adaptive mesh refinement hydrodynamics code},
The Astrophysical Journal Supplement Series 164 (2006), pp.~255--279.

\bibitem{ZhangShu2010maximum}
{\sc X. Zhang and C.-W. Shu},
{\em On maximum-principle-satisfying high order schemes for scalar conservation laws},
Journal of Computational Physics 229 (2010), pp.~3091--3120.


\bibitem{ZhangShu2010positivity}
{\sc X. Zhang and C.-W. Shu},
{\em On positivity-preserving high order discontinuous Galerkin schemes for compressible Euler equations on rectangular meshes.}
Journal of Computational Physics  229 (2010), pp.~8918--8934.


\bibitem{ZhangShu2011maximum}
{\sc X. Zhang and C.-W. Shu},
{\em Maximum-principle-satisfying and positivity-preserving high-order schemes for conservation laws: 
survey and new developments}, 
Proceedings of the Royal Society A 467 (2011), pp.~2752--2776.


\bibitem{ZhangShu2011positivity}
{\sc X. Zhang and C.-W. Shu},
{\em Positivity-preserving high order discontinuous Galerkin schemes for compressible Euler equations with source terms},
Journal of Computational Physics  230 (2011), pp.~1238--1248.


\bibitem{ZhangShu2012}
{\sc X. Zhang and C.-W. Shu},
{\em Positivity-preserving high order finite difference WENO schemes for compressible Euler equations},
Journal of Computational Physics  231 (2012), pp.~2245--2258.





\end{thebibliography}
\end{document}